\documentclass{amsart}
\usepackage[margin=1in]{geometry}
\usepackage{amsrefs}
\usepackage{amsmath}
\usepackage{amssymb}
\usepackage{amsthm}
\usepackage{amscd}
\usepackage{enumerate}

\usepackage{graphicx}
\usepackage{amsmath,amsthm,amsfonts,amscd,amssymb,comment,eucal,latexsym,mathrsfs}
\usepackage{stmaryrd}
\usepackage[all]{xy}

\usepackage{epsfig}

\usepackage[all]{xy}
\xyoption{poly}
\usepackage{fancyhdr}
\usepackage{wrapfig}
\usepackage{epsfig}

\theoremstyle{plain}
\newtheorem{thm}{Theorem}[section]
\newtheorem{prop}[thm]{Proposition}
\newtheorem{lem}[thm]{Lemma}
\newtheorem{cor}[thm]{Corollary}

\theoremstyle{definition}
\newtheorem{defn}[thm]{Definition}
\theoremstyle{remark}

  \def\C{{\mathbb{C}}}   \def\F{{\mathbb{F}}}        \def\N{{\mathbb{N}}}    \def\R{{\mathbb{R}}}  \def\T{{\mathbb{T}}}      \def\Z{{\mathbb{Z}}}

\newcommand\Aut{\operatorname{Aut}}

\newcommand\Fix{{\operatorname{Fix}}}

\newcommand\id{\operatorname{id}}

\renewcommand\Im{\operatorname{Im}}

\newcommand\Meas{{\operatorname{Meas}}}

\newcommand\Proj{\operatorname{Proj}}
\newcommand\Prob{\operatorname{Prob}}

\newcommand\rea{\operatorname{Re}}

\newcommand\supp{\operatorname{supp}}

\newcommand\Stab{\operatorname{Stab}}
\newcommand\Sub{\operatorname{Sub}}

\newcommand\topo{\operatorname{top}}

\newcommand{\actson}{\curvearrowright}
\newcommand{\actons}{\curvearrowright}
\newcommand{\acston}{\actson}

\newcommand{\ip}[1]{\langle #1 \rangle}

\begin{document}
\title[Max-min theorems, homoclinic points, completely positive entropy]{Max-Min theorems for weak containment, square summable homoclinic points, and completely positive entropy}      
\author{Ben Hayes}\thanks{The author gratefully acknowledges support from  NSF Grant DMS-1827376.}
\address{University of Virginia\\
          Charlottesville, VA 22904}
\email{brh5c@virginia.edu}

\begin{abstract}
    We prove a max-min theorem for weak containment in the context of algebraic actions. Namely, we show that given an algebraic action $G\actson X,$ there is a  maximal, closed $G$-invariant subgroup $Y$ of $X$ so that $G\actson (Y,m_{Y})$ is weakly contained in a Bernoulli shift. This subgroup is also the minimal closed subgroup so that any action weakly contained in a Bernoulli shift is $G\actson X/Y$-ergodic ``in the presence of $G\actson X$". We give several applications, including a major simplification of the proof that measure entropy equals topological entropy for principal algebraic actions whose associated convolution operator is injective. We also deduce from our techniques that algebraic actions whose square summable homoclinic group is dense have completely positive entropy when the acting group is sofic.
\end{abstract}

\maketitle

\textbf{Keywords:} Weak containment, homoclinic points, completely positive entropy.

\textbf{MSC:} 37A15, 37A35, 37A55,  22D25

\tableofcontents

\section{Introduction}

Let $G$ be a countable, discrete group. An \emph{algebraic action} of $G$ is an action $G\actson X$ by continuous automorphisms of a compact group $X.$ If we equip $X$ with the Haar measure $m_{X}$, this action becomes a probability-measure preserving action. We are typically interested in purely ergodic theoretic properties of $G\actson (X,m_{X})$ (e.g. ergodicity, mixing, complete positive entropy etc). However, the additional topological and algebraic structure that $X$ possesses provide additional tools to study the ergodic theoretic properties of $G\actson (X,m_{X})$ and one is able to efficiently establish an abstract theory as well as machinery around the ergodic theory of algebraic actions.

For example, in \cite{Me13} we showed, when $G$ is sofic, the Pinsker factor (properly defined) of an algebraic action is itself an algebraic action. In \cite{HayesLW*}, we gave a precise theoretical framework for studying the equality of topological and measure entropy for algebraic actions. Results in \cite{Me12} give a sufficient condition for this equality to occur, and our results in \cite{HayesLW*} give a complete solution of when this sufficient condition is satisfied. The solution provided in \cite{HayesLW*} is provided in terms of a max-min result inside a complete lattice naturally associated to the algebraic action.

One of the main results in this paper is to obtain a similar max-min result in the context of  \emph{weak containment} (as defined in Chapter II.10 of \cite{KechrisGA}) of actions. Given two probability measure-preserving actions $G\actson (Y,\nu),$ $G\actson (Z,\zeta),$ saying that $G\actson (Y,\nu)$ is weakly contained in $G\actson (Z,\zeta)$ roughly means that every ``finitary piece" of $G\actson (Y,\nu)$ can be approximated by a ``finitary piece of $G\actson (Z,\zeta)$". It is a weak form of saying that $G\actson (Y,\nu)$ is a factor of $G\actson (Z,\zeta),$ and many properties that pass through factor maps also pass through weak containment. We say that $G\actson (Z,\zeta),G\actson (Y,\nu)$ are weakly equivalent if each weakly contains the other. Weak containment is an area of research of significant current interest (see \cite{AasPopa, AbertWeiss, LewisRobinSWE, BurtonKechrisMax, BurKechWC, BurLupTam, AdrianRobinWCR, RobinWEClass}).

In \cite{MeWE}, we gave a class of algebraic actions which were weakly equivalent to Bernoulli shifts. These examples were all \emph{balanced} algebraic actions (i.e. the action of $G$ on the Pontryagin dual of $\Z(G)^{\oplus n}/\Z(G)^{\oplus n}f$ for some $f\in M_{n}(\Z(G))$). The study of \emph{balanced} algebraic actions includes the case of \emph{principal} algebraic actions and has been an active area of research in recent years.  Our first major result in this paper is to give a precise max-min principle which says that given \emph{any} algebraic action, we can find a maximal, $G$-invariant closed subgroup $Y$ of $X$ so that $G\actson (Y,m_{Y})$ is weakly contained in a Bernoulli shift. The subgroup $Y$ also ends up being the \emph{minimal} closed subgroup satisfying a certain ergodicity property. Given a probability measure-preserving action $G\actson (Z,\zeta)$ and a Borel action $G\acston Y$ on a standard Borel space we say that $G\actson (Z,\zeta)$ is $G\actson Y$-\emph{ergodic} if for every $G$-equivariant, measurable $\Theta\colon Z\to Y$ we have that $\Theta$ is almost surely constant. We actually need a slight weakening of this. Suppose that $W$ is a standard Borel space, that $G\actson W$ is a Borel action, and that $G\actson Y$ is a \emph{Borel} factor with factor map $q\colon W\to Y.$ We say that $G\acston (Z,\zeta)$ is \emph{$G\actson Y$-ergodic in the presence of $G\actson W$} if for every $G$-equivariant, measurable map $\Theta\colon Z\to W$ we have that $q\circ \Theta$ is almost surely constant. Said differently: every equivariant, measurable map $Z\to Y$ is almost surely constant \emph{provided it has a lift to an equivariant map $Z\to W.$} With this terminology, we can now state the main result of this paper.

\begin{thm}\label{T:main theorem max min}
Let $G$ be a countable, discrete group and $G\actson  X$ an algebraic action.
\begin{enumerate}[(i)]
\item There is a largest closed, $G$-invariant subgroup $Y$ of $X$ so that $G\actson (Y,m_{Y})$ is weakly contained in a Bernoulli shift. \label{I:max thing intro}
\item If $Y$ is as in (\ref{I:max thing intro}), then $Y$ may be characterized as the smallest closed, $G$-invariant subgroup of $X$ with the property that every action weakly contained in a Bernoulli shift is $G\actson X/Y$-ergodic in the presence of $G\actson X.$
\end{enumerate}
\end{thm}

We remark that one can easily formulate what it means for a probability measure-preserving action $G\actson (X,\mu)$ to be weakly contained in a class $\mathcal{C}$ of actions.  Theorem \ref{T:main theorem max min} also works for many other classes of actions, provided this class has the property that it is closed under weak containment and products. In this case, Theorem \ref{T:main theorem max min} holds verbatim with every instance of a Bernoulli shift replaced by an action in the class $\mathcal{C}$ (see Corollary \ref{C:structure corollary}). Our first application  of Theorem \ref{T:main theorem max min} is the following result, showing that many of the concrete algebraic actions that have been of relevance to entropy theory are weakly contained in Bernoulli shifts. We let $\Z(G)$ denote the integral group ring of $G.$ If $f\in M_{m,n}(\Z(G)),$ then $f$ naturally induces a convolution operator
$\lambda(f)\colon \ell^{2}(G)^{\oplus n}\to \ell^{2}(G)^{\oplus m}$
(see Section \ref{S:extend convolve} for the precise formula).

\begin{thm}\label{T:weak containment intro}
Let $G$ be a countable, discrete, group and let $f\in M_{n}(\Z(G))$ be such that $\lambda(f)$ is injective on $\ell^{2}(G)^{\oplus n}.$ Then $G\actson (X_{f},m_{X_{f}})$ is weakly contained in a Bernoulli shift.
\end{thm}

For example, if $f\in \Z(G)$ and is injective as a left convolution operator on $\ell^{2}(G)$, then $G\actson (X_{f},m_{X_{f}})$ is weakly contained in  a Bernoulli shift. If $G$ is torsion-free, then the Atiyah conjecture would predict that \emph{every} nonzero element of $\Z(G)$ is injective as a left-convolution operator on $\ell^{2}(G).$ We refer the reader to \cite[Chapter 10]{Luck} for examples of groups which satisfy the Atiyah conjecture.
The requirement that every nonzero element of $\Z(G)$ is injective as a left convolution operator on $\ell^{2}(G)$ is also (a priori) weaker than  requiring that $G$ be torsion-free and satisfy the Atiyah conjecture. For example, this is known for every left-orderable group (see \cite{LinnellZDC}). Left-orderable groups include polycyclic groups, certain groups of
intermediate growth (\cite{GRMakiOrder}), Thompson’s group, certain mapping class groups of connected surfaces with boundary (\cite{OrderMCG}),
and is closed under free products (see \cite[Theorem 2.7]{Passman}) . It is also direct to show that extensions of left-orderable groups by
left-orderable groups are left-orderable. Recent results (see \cite{HydeLodLeftOrd}, also \cite[Theorem 1.7]{KimKobLodLeftOrd})  show that there is a continuum family
$(G_{\alpha})_{\alpha}$ of pairwise non-isomorphic, simple, finitely generated, left orderable groups.

Knowing that an action is weakly contained in a Bernoulli shift is interesting in its own right. However, an additional significant consequence of Theorem \ref{T:weak containment intro} is a heavy simplification of part of the proof of the main result of \cite{Me5}. Note that being weakly contained in a Bernoulli shift implies strong soficity (as defined implicitly in \cite{AustinAdd} and explicitly in \cite[Definition 3.4]{Me13}) with respect to any sofic approximation. So Theorem \ref{T:weak containment intro} in particular implies that $G\actson (X_{f},m_{X_{f}})$ is strongly sofic. Section 5 of \cite{Me5} proves that $G\actson (X_{f},m_{X_{f}})$ is strongly sofic, but the proof there is the most technical portion of an already technical paper. Our proof of Theorem \ref{T:weak containment intro} is much easier than what is in \cite{Me5}, though it is certainly inspired by that proof. By \cite[Theorem 1.1]{Me12}, the fact that $G\actson (X_{f},m_{X_{f}})$ is strongly sofic allows one to show that the measure entropy of $G\actson (X_{f},m_{X_{f}})$ equals $\log \det_{L(G)}(f),$ only knowing that the topological entropy of $G\actson X_{f}$ is the logarithm of the Fuglede-Kadison determinant. The case of topological entropy is handled in Sections 3-4 of \cite{Me5} and is much easier to establish than the measure entropy case. So Theorem \ref{T:weak containment intro} can be used to give a major simplification of the proof that the measure entropy of $G\actson (X_{f},m_{X_{f}})$ is the logarithm of the Fuglede-Kadison determinant, which is a significant and important result in the field of sofic entropy as well as the study of algebraic actions.

As part of the proof of Theorem \ref{T:weak containment intro}, we establish a way to associate to every $m,n\in \N,$ $\xi\in M_{n,m}(\ell^{2}(G,\R)),$ and every $\nu\in \Prob(\R^{m})$ with mean zero and finite second moment, a $G$-equivariant, measurable map $\Theta_{\xi}\colon (\R^{m})^{G}\to (\T^{n})^{G}.$ This technique may be of independent interest. When $m=n=1,$ and $\xi\in c_{c}(G,\R),$ this map is simply given by right convolution (and there is a similar modification when one of $m,n>1$). We are able to extend the definition of $\Theta_{\xi}$ to the case of $\xi\in \ell^{2}(G,\R)$  by using a uniform continuity argument, as well as the completeness of both $\ell^{2}(G,\R)$ and the space of measurable maps $\R^{G}\to \T^{G}$ in the topology of convergence in measure with respect to $\nu^{\otimes G}$. See Section \ref{S:extend convolve} for the precise details.  This is similar to the results in \cite{MeWE}, and this idea of using convolution to prove weak containment goes back to \cite{BowenEntropy} but new arguments are required to extend convolution measurable to the case of $\xi\in \ell^{2}(G,\R).$  The work in \cite{BowenEntropy}  uses $\Theta_{\xi}$ for $\xi\in \ell^{1}(G,\R),$ and the definition in this case is much more transparent. If $\xi\in \ell^{1}(G,\R),$  then its image in $\T^{G}$ is called a \emph{summable homoclinic point}. Arguments involving convolution by a summable homoclinic point have a long history in the study of algebraic actions going back to Lind-Schmidt in \cite{LindSchmidtHomoc}, and similar arguments were also used to great effect in \cite{LindSchmidt1, LSVHomoc}. To the best of our knowledge, the work in \cite{MeWE} is the first that uses $\ell^{2}$ vectors instead of $\ell^{1}$ vectors in the context of weak containment. The results in \cite{MeWE} are about probability measures on $(\T^{n})^{G}$ associated to vectors in $M_{n}(\ell^{2}(G,\R)).$
Those results by themselves can be modified to prove Theorem \ref{T:weak containment intro}. However, the ability to exhibit these probability measures on $(\T^{n})^{G}$ as factors of Bernoulli measures enables us to effortlessly prove results on complete positive entropy of algebraic actions.

By entropy in this context, we mean measure entropy as defined by Bowen in \cite{Bow} (and in full generality by Kerr-Li in \cite{KLi}) for actions of sofic groups. The class of sofic groups is a large class which include all amenable groups, all linear groups, and is closed under free products with amalgamation over amenable subgroups, as well as all wreath products (see  \cite{DKP, HayesSale, LPaun, PoppArg}). Entropy for actions of sofic groups agrees with the usual entropy defined by Kieffer (see \cite{Kieff}) when the group is amenable,  by the results in \cite{BowenAmen,KLi2}. Sofic groups are the largest class of groups where it is known that one can define entropy so that entropy of a Bernoulli shift is equal to the entropy of the base. Sofic groups are also the largest class of groups where it is known that entropy can be defined to be a conjugacy invariant which distinguishes Bernoulli shifts with different base entropies. So the fact that our results show complete positive entropy for the class of actions of sofic groups should be taken to be optimal. We mention here that it is not known if all groups are sofic.

Recall that if $G\actson X$ is an algebraic action then the homoclinic group of $X,$ denoted $\Delta(G\actson X),$ is the subgroup of $X$ consisting of all $x\in X$ so that $\lim_{g\to\infty}gx=0.$ If $X$ is abelian, and $1\leq p<\infty,$ Chung-Li in \cite[Section 5]{ChungLi} defined the $p$-summable homoclinic group to be the set of $x\in X$ so that $\sum_{g\in G}|\chi(gx)|^{p}<\infty$ for all $\chi\in\widehat{X}.$
Here $|x+\Z|=\inf_{n\in \Z}|x+n|$ for all $n\in \Z.$ We let $\Delta^{(p)}(G\actson X)$ be the $p$-summable homoclinic group of $G\actson X.$ It is also easy to see that the $p$-summable homoclinic group is contained in the homoclinic group. It is easy to see that a square summable homoclinic point in $\T^{G}$ is the image of an element of $\ell^{2}(G,\R)$ under the canonical quotient map $\R^{G}\to \T^{G}.$ As mentioned before, associated to every $\xi\in \ell^{2}(G,\R),$ and to every $\nu\in \Prob(\Z)$ which has mean zero and a finite second moment, we have a $\nu^{\otimes G}$ measurable map $\Theta_{\xi}\colon \Z^{G}\to \T^{G}$ which is a measurable extension of convolving by the adjoint of $\xi$. The maps $(\Theta_{\xi})_{\xi\in \ell^{2}(G,\R)}$ along with Theorem \ref{T:main theorem max min} allow one to prove the following new result on complete positive entropy.

\begin{thm}\label{T:new theorem cpe}
Let $G$ be a countably infinite, discrete, group and let $G\actson X$ be an algebraic action with $X$ abelian. Suppose that $\Delta^{(2)}(G\actons X)$ is dense in $X.$ Let $N$ be the kernel of $G\actons X.$ Then the induced action $G/N\actson (X,m_{X})$ is weakly equivalent to a Bernoulli shift.
If $G$ is sofic, then $G\actson (X,m_{X})$ has completely positive entropy in the presence with respect to any sofic approximation of $G.$

\end{thm}

We remark that the results in Section \ref{S:extend convolve} rely on the fact that if $\nu\in \Prob(\R)$ has mean zero and a finite second moment, then the Fourier transform of $\nu$, denoted $\widehat{\nu},$ is $C^{2}$ with $\widehat{\nu}(0)=1,\widehat{\nu}'(0)=0.$ This implies that $\widehat{\nu}(t)=1+O(t^{2})$ as $t\to 0,$ and this is crucially what we use to define $\Theta_{\xi}$ as well as to compute $(\Theta_{\xi})_{*}(\nu^{\otimes G}).$ If we wanted to extend $\Theta_{\xi}$ to a large class of vectors, say $\xi \in \ell^{p},$ we would need to consider probability measures $\nu$ with the property that $\widehat{\nu}(t)=1+O(t^{p}).$ It is a well known fact once $p>2$ \emph{the only such measure is the dirac mass at $0$}. In fact,
\[\rea\left(\frac{1-\widehat{\nu}(t)}{t^{2}}\right)=2\int \left(\frac{\sin (\pi tx)}{t}\right)^{2}\,d\nu(x),\mbox{ for all $t\ne 0.$}\]
By Fatou's Lemma, the above shows that if $\widehat{\nu}(t)=1+O(t^{p})$ with $p>2,$ then $\int |x|^{2}\,d\nu(t)=0,$ and so $\nu=\delta_{0}.$
Thus we cannot use the same arguments to extend $\Theta_{\xi}$ to the case $\xi\in \ell^{p}(G)$ with $p>2$ in a way that $\Theta_{\xi}$ varies continuously in $\xi.$ For this reason, it is unlikely that one can prove Theorem \ref{T:new theorem cpe} by only assuming that $\Delta^{(p)}(G\actson X)$ is dense in $X$ if $p>2.$ See Propositions \ref{P:cant extend}, \ref{P:cant extend seconds} for more detailed results.

Suppose that  $f\in M_{n}(\Z(G))$ and $\lambda(f)\colon \ell^{2}(G)^{\oplus n}\to \ell^{2}(G)^{\oplus n}$ is the associated convolution operator (see Section \ref{S:extend convolve} for the precise definitions). If $f$ has an $\ell^{2}$ formal inverse in the sense of \cite{MeWE}, then $G\actson X_{f}$ has a square summable homoclinic point. Additionally if $\lambda(f)$ is invertible, then $f$ has an $\ell^{2}$ formal inverse. So
Theorem \ref{T:new theorem cpe} is a common generalization of both \cite[Corollary 1.5]{Me13} and \cite[Theorem 1.1]{MeWE}. A key difference between Theorem \ref{T:new theorem cpe} and \cite[Corollary 1.5]{Me13} is that Theorem \ref{T:new theorem cpe} applies not only to actions taking a very specific form, such as balanced algebraic actions, but gives an abstract criterion that one can check to show that an action has completely positive entropy. Moreover, we are able to associate to \emph{any} closed, $G$-invariant subspace $\mathcal{H}$ of $\ell^{2}(G)^{\oplus n}$ an algebraic subshift $X^{\mathcal{H}}$ of $(\T^{G})^{\oplus n}$ so that $X^{\mathcal{H}}$ has dense square summable homoclinic group (see Definition \ref{D:homoclinc assoicated to a rep}). So Theorem \ref{T:new theorem cpe} applies to a very large class of actions.

Square summable homoclinic points include summable homoclinic points and coincide with homoclinic points if $X$ is totally disconnected. So we obtain the following corollary for free.

\begin{cor}\label{C:easy stuff intro}
Let $G$ be a countably infinite, discrete, group and let $G\actson X$ be an algebraic action.
If $X$ is totally disconnected and abelian, and $\Delta(G\actson X)$ is dense in $X,$ then $G\actson (X,m_{X})$ is weakly equivalent to a Bernoulli shift. If $G$ is sofic, then $G\actson (X,m_{X})$  has completely positive entropy in the presence with respect to any sofic approximation of $G.$
\end{cor}

We remark that actions of the above type were already considered in the nonamenable context in \cite{GabSew}, though in \cite{GabSew} they allow $X$ to be nonabelian. Our results show that the actions considered there have completely positive entropy in the presence, provided $X$ is abelian. In \cite{GabSew} it is implicitly shown that $G\actson (X,m_{X})$ is strongly sofic, and Corollary \ref{C:easy stuff intro} implies strong soficity when $X$ is abelian with a different proof. By the main theorem of \cite{Me12}, Corollary \ref{C:easy stuff intro}  also recovers \cite[Theorem 8.2]{GabSew} when $H$ is abelian.

A crucial step in the proof Theorem \ref{T:new theorem cpe} is the ability to reduce to the case that $X$ is generated by the image of a  single $\ell^{2}$-vector. When trying to show that $G\actson (X,m_{X})$ is weakly contained in a Bernoulli shift, this reduction follows trivially
from Theorem \ref{T:weak containment intro}. However, for the question of complete positive entropy, we need the following new result.

\begin{thm}\label{T:join of cpe is cpe intro}
Let $G$ be a countable, discrete, sofic group with sofic approximation $(\sigma_{k})_{k}.$ Let $G\actson X$ be an algebraic action. Suppose that $(X_{j})_{j\in J}$ are closed, $G$-invariant subgroups of $X,$ which topologically generate $X.$ If for every $j\in J$ we have that $G\actson (X_{j},m_{X_{j}})$ is strongly sofic with respect to $(\sigma_{k})_{k}$ and has completely positive entropy in the presence, then $G\actson (X,m_{X})$ is strongly sofic and has completely positive entropy in the presence.
\end{thm}
By ``topologically generate $X$" we mean that the smallest closed subgroup of $X$ containing all the $X_{j}$ is $X$ itself.

We close the introduction by discussing the organization of the paper. In Section \ref{S:weak containment} we prove Theorem \ref{T:main theorem max min}, as well as a general version for weak containment with respect to other classes of actions. In Section \ref{S:extend convolve} we define, in a natural way, an equivariant map $\Theta_{\xi}\colon (\R^{m})^{G}\to (\T^{G})^{\oplus n}$ associated to any $\xi\in M_{m,n}(\ell^{2}(G,\R))$ and any probability measure on $\R^{m}$ which has a finite second moment and is mean zero. This map extends convolution in the case that $\xi$ is finitely supported. In Section \ref{S:appplications}, we apply this measurable extension of convolution, along with our main theorem (Theorem \ref{T:main theorem max min}) to get many new results on weak containment, and on complete positive entropy. These include Theorems \ref{T:weak containment intro},\ref{T:new theorem cpe}, \ref{T:join of cpe is cpe intro}. The study of weak containment and weak equivalence is closely related to the study of random stabilizers, i.e. the study of invariant random subgroups, for this reason in Appendix \ref{S:IRS} we give a classification of the types of invariant random subgroups that can arise from algebraic actions.

\textbf{Acknowledgments.}
I thank Robin Tucker-Drob, Doug Lind, Anush Tserunyan, and Mehrdad Kalantar for interesting discussions related to this topic. I thank Robin Tucker-Drob for pointing me to the correct reference for invariant subgroups of algebraic actions. I thank Doug Lind for pointing out some specific interesting corollaries of the complete positive entropy results. I thank Peter Burton and Lewis Bowen for comments on an earlier draft of this article. Part of this work was done while I was visiting the State University of New York at Buffalo. I thank the State University of New York at Buffalo for their hospitality and providing a productive work environment. I thank the anonymous referee for their numerous comments, which greatly improved the paper.

\subsection{Notational Conventions}

In order to work in the proper generality, we will need to adopt some notation for product spaces. If $m\in \N,$ and $A$ is a set, then $A^{m}$ will be regarded as the set of all functions $\{1,\cdots,m\}\to A.$ We also use $A^{\infty}$ for $A^{\N}.$ If $J$ is a set, we let $c_{c}(J,\C)$ be all finitely supported functions $J\to \C,$ with similar notation for $c_{c}(J,\R),c_{c}(J,\Z)$ etc.

If $(A,\Sigma)$ is a standard Borel space, we let $\Prob(A)$ denote the space of completed, Borel probability measures on $A.$ We will often drop $\Sigma$ from the notation if it is clear from context or not necessary. If $A$ is a Polish space, then we equip $A$ with the obvious Borel structure. A \emph{Lebesgue space} will be a complete probability space $(X,\mu)$ which is isomorphic modulo null sets to a Polish space equipped with a completed Borel probability measure.

If $X$ is a Hausdorff topological group, we say that $S\subseteq X$ \emph{topologically generates $X$} if the smallest closed subgroup of $X$ containing $S$ is $X$ itself. If $(Y_{j})_{j\in J}$ are closed subgroups of $X,$ we use $\bigvee_{j\in J}Y_{j}$ for the smallest closed subgroup of $X$ containing all the $Y_{j}.$ If $X$ is  a Polish group,  we let $\Sub(X)$ be the space of closed subgroups of $X.$ We equip $\Sub(X)$ with the Chabauty topology. We use the notation $Y\leq X$ to mean that $Y$ is a closed subgroup of $X.$ Note that if $\mu,\nu\in \Prob(X),$ then there is a unique $\mu*\nu\in \Prob(X)$ so that
\[(\mu*\nu)(E)=\mu\otimes \nu(\{(x,y):xy\in E\})\]
for all Borel $E\subseteq X.$ Given $\mu\in \Prob(X),$ we define $\mu^{*}\in \Prob(X)$ by $\mu^{*}(E)=\mu(\{x:x^{-1}\in E\}).$ By a \emph{representation} of $X,$ we shall always mean a continuous homomorphism $\pi\colon X\to \mathcal{U}(\mathcal{H})$ where $\mathcal{H}$ is a Hilbert space, and $\mathcal{U}(\mathcal{H})$ is the group of unitaries on $\mathcal{H}.$ We shall say $\pi$ is finite-dimensional if $\mathcal{H}$ is. If $\pi\colon X\to \mathcal{U}(\mathcal{H})$ is a representation of $X$ and $\mu\in \Prob(X),$ we let $\pi(\mu)$ be the unique bounded operator on $\mathcal{H}$ so that
\[\ip{\pi(\mu)\xi,\eta}=\int_{X}\ip{\pi(x)\xi,\eta}\,d\mu(x)\mbox{ for all $\xi,\eta\in\mathcal{H}.$}\]
It is straightforward to see that $\pi(\mu*\nu)=\pi(\mu)\pi(\nu)$ and $\pi(\mu^{*})=\pi(\mu)^{*}$ for all $\mu,\nu\in \Prob(X).$

\section{The main result on weak containment}\label{S:weak containment}

In this section, we prove Theorem \ref{T:main theorem max min}. We will also prove a more general result for weak containment with respect to a fairly arbitrary class of actions. The first step is the following formulae which tell us how to recover the Haar measure on $Y_{1}\vee Y_{2},$ as well as the Haar measure on the group topologically generated by the support of a given probability measure. We use the following notation: given orthogonal projections $P,Q$ on a Hilbert space $\mathcal{H},$ we let $P\wedge Q$ be the orthogonal projection onto $P(\mathcal{H})\cap Q(\mathcal{H}).$

\begin{lem}\label{L:PW stuff}
Let $X$ be a compact group.
\begin{enumerate}[(i)]
\item For $Y_{1},Y_{2}\in \Sub(X),$ we have that $m_{Y_{1}\vee Y_{2}}=\lim_{n\to\infty}(m_{Y_{1}}*m_{Y_{2}}*m_{Y_{1}})^{*n}.$  \label{I:meet of subgroups}
\item For $\mu\in \Prob(X),$ we have that $m_{\overline{\ip{\supp(\mu^{*}*\mu)}}}=\lim_{n\to\infty}(\mu^{*}*\mu)^{*n}.$ \label{I:recovering the support}
\end{enumerate}

\end{lem}

\begin{proof}

(\ref{I:meet of subgroups}): By the Peter-Weyl theorem, it suffices to show that $\pi(m_{Y_{1}\vee Y_{2}})=\lim_{n\to\infty}(\pi(m_{Y_{1}})\pi(m_{Y_{2}})\pi(m_{Y_{1}}))^{n}$ for every finite-dimensional representation $\pi$ of $X.$ So fix a finite-dimensional representation $\pi$ of $X,$ and set $P_{j}=\pi(m_{Y_{j}})$ for each $j=1,2.$ Observe that each $P_{j}$ is an orthogonal projection. Since $\|P_{1}P_{2}P_{1}\|\leq 1,$ and $P_{1}P_{2}P_{1}\geq 0,$ the spectral theorem shows that
$(P_{1}P_{2}P_{1})^{n}$ converges to the projection onto the fixed points of $P_{1}P_{2}P_{1}.$ We claim that the projection onto the fixed points of $P_{1}P_{2}P_{1}$ is $P_{1}\wedge P_{2}.$ Clearly $P_{1}(\mathcal{H})\cap P_{2}(\mathcal{H})$ is  contained in the set of fixed points of $P_{1}P_{2}P_{1}.$ Conversely, suppose that $\xi\in \mathcal{H}$ is fixed by $P_{1}P_{2}P_{1}.$ Then $\|\xi\|=\|P_{1}P_{2}P_{1}\xi\|\leq \|P_{1}\xi\|\leq \|\xi\|,$ so $\|P_{1}\xi\|=\|\xi\|,$ and this implies that $P_{1}\xi=\xi.$ Similarly, we have that $P_{2}\xi=\xi,$ so the fixed points of $P_{1}P_{2}P_{1}$ are $P_{1}(\mathcal{H})\cap P_{2}(\mathcal{H}).$ It simply remains to show that $P_{1}\wedge P_{2}=\pi(m_{Y_{1}\vee Y_{2}}),$ and this is straightforward from the fact that $P_{j}$ is the orthogonal projection onto the $Y_{j}$-invariant vectors for $j=1,2$ and that $\pi(m_{Y_{1}\vee Y_{2}})$ is the projection onto the $Y_{1}\vee Y_{2}$-invariant vectors.

(\ref{I:recovering the support}): Fix a finite-dimensional representation $\pi\colon X\to \mathcal{U}(\mathcal{H}).$ As in (\ref{I:meet of subgroups}), it suffices to show that
\[\pi(m_{\overline{\ip{\supp(\mu^{*}*\mu)}}})=\lim_{n\to\infty}[\pi(\mu)^{*}\pi(\mu)]^{n}.\]
Since $1\geq \pi(\mu)^{*}\pi(\mu)\geq 0,$ it follows by the spectral theorem that $\lim_{n\to\infty}[\pi(\mu)^{*}\pi(\mu)]^{n}$ converges to the projection onto the fixed points of $\pi(\mu)^{*}\pi(\mu)=\pi(\mu^{*}*\mu).$ Set $\nu=\mu^{*}*\mu.$ It now simply suffices to show that the fixed points of $\pi(\nu)$ are $\pi(m_{\overline{\ip{\supp(\nu)}}})(\mathcal{H}).$ Let $P$ be the projection onto the fixed points of $\pi(\nu).$

Set $Y=\overline{\ip{\supp(\nu)}},$ then $\pi(m_{Y})$ is the projection onto the $Y$-fixed points in $\mathcal{H},$ and so it is clear that $\pi(m_{Y})\leq P.$ Conversely, suppose that $\xi\in \mathcal{H}$ and is fixed by $\pi(\nu).$ Then:
\[\|\xi\|^{2}=\rea(\ip{\pi(\nu)\xi,\xi})=\int_{X}\rea(\ip{\pi(x)\xi,\xi})\,d\nu(x).\]
By the Cauchy-Schwartz inequality we have that $\rea(\ip{\pi(x)\xi,\xi})\leq \|\xi\|^{2},$ and since $x\mapsto \rea(\ip{\pi(x)\xi,\xi})$ is continuous the above equation is only possible if $\rea(\ip{\pi(x)\xi,\xi})=\|\xi\|^{2}$ for every $x\in\supp(\nu).$ Since $\|\pi(x)\xi-\xi\|^{2}=2\|\xi\|^{2}-2\rea(\ip{\pi(x)\xi,\xi})$ for all $x\in X,$ we have that $\pi(x)\xi=\xi$ for all $x\in \supp(\nu).$ But $\{x\in X:\pi(x)\xi=\xi\}$ is a closed subgroup of $X,$ so we must have that $\pi(x)\xi=\xi$ for all $x\in Y.$ Thus $P\leq \pi(m_{Y}),$ so $P=\pi(m_{Y}).$

\end{proof}

Recall that $\Sub(X)$ is a partially ordered set under the order $Y\leq X$ if $Y\subseteq X.$ If $(I,\preceq)$ is any partially ordered set, then an element $i\in I$ is said to be the \emph{largest} element of $I$ if $j\preceq i$ for every $j\in I.$ It may not be the case that largest elements exist. Note that this is stronger than being a \emph{maximal} element of $I,$ which is an element $i\in I$ so that if $j\in I$ and $j\succeq i,$ then $i=j.$ Similarly, one defines a \emph{smallest} element, as well as a \emph{minimal} element of a partially ordered set.
\begin{lem}\label{L:maximality argument}
Let $X$ be a compact group, and let $\mathcal{P}\subseteq \Prob(X)$ be closed under convolutions, the $*$-operation, and in the weak$^{*}$ topology. Let $\mathcal{S}=\{Y\in \Sub(X):m_{Y}\in \mathcal{P}\}.$
\begin{enumerate}[(i)]
\item For every $Y_{1},Y_{2}\in \mathcal{S},$ we have that $Y_{1}\vee Y_{2}\in \mathcal{S}.$  \label{I:group generated by support}
\item There is largest element $Y\in \mathcal{S}$ with respect to the containment order. Further,   $\supp(\nu^{*}*\nu)\subseteq Y$ for every $\nu\in \mathcal{P}.$ \label{I:uniqueness semigroup containment}
\item If $Y$ is as in (\ref{I:uniqueness semigroup containment}), then for every $\nu\in \mathcal{P}$ there is a $c\in X/Y$ so that $\supp(\nu)\subseteq c.$ \label{I:coset support}
\end{enumerate}
\end{lem}

\begin{proof}

(\ref{I:group generated by support}):
This follows from Lemma \ref{L:PW stuff} (\ref{I:meet of subgroups}).

(\ref{I:uniqueness semigroup containment}):
By (\ref{I:group generated by support}) and the fact that $\mathcal{S}$ is closed in the Chabauty topology, there is a largest element $Y\in \mathcal{S}.$ Suppose that $\nu\in \mathcal{P}.$ By Lemma \ref{L:PW stuff} (\ref{I:recovering the support}),
\[m_{\overline{\ip{\supp(\nu^{*}*\nu)}}}=\lim_{n\to\infty}(\nu^{*}*\nu)^{*n}\in \mathcal{P}.\]
Thus $\overline{\ip{\supp(\nu^{*}*\nu)}}\in \mathcal{S},$ so $\overline{\ip{\supp(\nu^{*}*\nu)}}\subseteq Y.$ Thus $Y$ is the desired element of $\mathcal{S}.$

(\ref{I:coset support}):
Since $\nu\in \mathcal{P},$ we have that $\supp(\nu^{*}*\nu)\subseteq Y$ by  (\ref{I:uniqueness semigroup containment}). Let $f\colon X\times X\to X$ be the map $f(x,y)=x^{-1}y,$ so $\nu^{*}*\nu=f_{*}(\nu\otimes \nu).$ Let $q\colon X\to X/Y$ be the quotient map. Then
\[\supp((q\circ f)_{*}(\nu\otimes \nu))=q(\supp(f_{*}(\nu\otimes \nu)))=q(\supp(\nu^{*}*\nu))=\{Y\}.\]
We thus have that $q\circ f$ is $\nu\otimes\nu$-almost surely equal to $Y.$ Thus for $\nu\otimes\nu$-almost every $(x_{1},x_{2})\in X\times X$ we have that $x_{1}Y=x_{2}Y.$ By Fubini's theorem, this implies that there is a $c\in X/Y$ so that $xY=c$ for $\nu$-almost every $x\in X.$ Thus $\supp(\nu)\subseteq c.$

\end{proof}

If $(X,\mu)$ is a probability space, then a \emph{finite observable} is, by definition, a measurable map $\alpha\colon X\to A$ where $A$ is a finite set equipped with the $\sigma$-algebra of all subsets of $A.$
Suppose that $G$ is a countable, discrete, group and $G\actson (X,\mu)$ is a probability measure-preserving action. Given a finite observable $\alpha\colon X\to A,$ and a finite $F\subseteq G,$ we let $\alpha^{F}\colon X\to A^{F}$ be given by $\alpha^{F}(x)(g)=\alpha(g^{-1}x)$ for $g\in F,x\in X.$

\begin{defn}
Let $G$ be a countable, discrete, group, and $\mathcal{C}$ a class of probability measure-preserving actions of $G.$ Given a probability measure-preserving action $G\actson (Y,\nu),$ we say that $G\actson (Y,\nu)$ is \emph{weakly contained in $\mathcal{C}$} if for every $\varepsilon>0,$ every finite $F\subseteq G,$ and every finite observable $\alpha\colon Y\to A,$ there is a probability measure-preserving action $G\actson (X,\mu)$ in $\mathcal{C}$ and a finite observable $\beta\colon X\to A$ so that
\[\|(\alpha^{F})_{*}(\nu)-(\beta^{F})_{*}(\mu)\|_{1}<\varepsilon.\]
In the above inequality we are identifying $\Prob(A^{F})$ with the subset of $\ell^{1}(A^{F})$ consisting of functions which are nonnegative and which have $\|\cdot\|_{1}$-norm $1.$
\end{defn}

Suppose in the preceding definition that $Y$ is a compact metrizable space, that the action is by homeomorphisms, and that the measure $\nu$ is the completion of a Borel probability measure. Then it is easy to see that $G\actson (Y,\nu)$ is weakly contained in $\mathcal{C}$ if and only if there is a sequence $G\actson (X_{n},\mu_{n})$ of actions in $\mathcal{C}$ and measurable maps $\psi_{n}\colon X_{n}\to Y$ so that:
\begin{itemize}
\item $\lim_{n\to\infty}(\psi_{n})_{*}(\mu_{n})=\nu$ \mbox{ in the weak$^{*}$ topology,}
\item for all $g\in G,$ $\mu_{n}(\{x\in X_{n}:(\psi_{n}(gx),g\psi_{n}(x))\in \mathcal{O}\})\to 1$ for every open neighborhood $\mathcal{O}$ of the diagonal in $Y\times Y.$
\end{itemize}

If $A$ is a set, and $G$ is a group, we always have the \emph{left shift action} $G\actson A^{G}$ given by
\[(ga)(h)=a(g^{-1}h)\mbox{ for $g,h\in G,a\in A^{G}$.}\]
We will occasionally also use the \emph{right shift action} $\rho$ given by
\[(\rho(g)a)(h)=a(hg)\mbox{ for $g,h\in G,a\in A^{G}.$}\]
The left shift action will be what we use more often, thus for $a\in A^{G},g\in G,$ the notation $ga$ will be reserved for the action of $g$ on $a$ under the left shift action. Also, when we write $G\actson A^{G}$ we will, unless otherwise  stated, mean the left shift action of $G$ on $A^{G}.$ If $A$ is a Borel space, and $\mu\in \Prob(A),$ the measure $\mu^{\otimes G}$ is invariant under the left shift action. The action $G\actson (A^{G},\mu^{\otimes G})$ will be called the \emph{Bernoulli shift action}.

\begin{defn}
Let $G$ be a countable, discrete, group and let $\mathcal{C}$ be a class of probability measure-preserving actions of $G.$ We say that $\mathcal{C}$ is \emph{weakly closed under products} if whenever $G\actson (X_{1},\mu_{1}),G\actson (X_{2},\mu_{2})$ are in $\mathcal{C},$ then $G\actson (X_{1}\times X_{2},\mu_{1}\otimes \mu_{2})$ is weakly contained in $\mathcal{C}.$

\end{defn}

Here are some examples of classes which are weakly closed under products:
\begin{enumerate}
\item the class of all \emph{sofic} actions (trivial from the fact that the product of two sofic actions is sofic),
\item the class of all Bernoulli actions,
\item   the one-element class consisting of a fixed nontrivial Bernoulli action, when $G$ is infinite (by the Ab\'{e}rt-Weiss result \cite{AbertWeiss}),
\item the class of all translation actions: i.e actions of the form $G\actson (X,m_{X})$ where $X$ is a compact Hausdorff group, and $G$ acts by $gx=\phi(g)x$ for some homomorphism $\phi\colon G\to X,$
\item the class of all \emph{finite} actions, i.e. actions of the form $G\actson (A,\mu)$ where $A$ is a finite set and $\mu$ is $G$-invariant.
\end{enumerate}

\begin{cor}\label{C:structure corollary}
Let $X$ be a compact, metrizable group and let $G$  be a countable, discrete group with $G\actson X$ by continuous automorphisms.  Fix a class $\mathcal{C}$ of probability measure-preserving actions which is weakly closed under products.

\begin{enumerate}
\item There is a largest closed, $G$-invariant subgroup $Y$ of $X$ so that $G\actons (Y,m_{Y})$ is weakly contained in the class $\mathcal{C}.$ \label{I:maximal property}
\item We may characterize $Y$ in  (\ref{I:maximal property}) as the smallest closed, $G$-invariant, subgroup of $X$ with the following property: every $G\actson (Z,\zeta)$ weakly  contained in $\mathcal{C}$ is $G\actson X/Y$-ergodic in the presence of $G\actson X.$
\label{I:minimal property}
\end{enumerate}
\end{cor}

\begin{proof}

(\ref{I:maximal property}): Let $\mathcal{P}$ be the set of $\mu\in \Prob(X)$ so that $G\actson (X,\mu)$ is weakly contained in $\mathcal{C}.$ Then $\mathcal{P}$ is clearly weak$^{*}$ closed. It is also closed under convolution: if $\mu,\nu\in \mathcal{P},$ then $G\actson (X,\mu*\nu)$ is a factor of the product action $G\actson (X\times X,\mu\otimes \nu)$ with factor map $p\colon X\times X\to X$ given by $p(x,y)=xy.$ We also have that $\mathcal{P}$ is closed under the $*$ operation, since we have an isomorphism of probability measure-preserving actions $G\actson (X,\mu)\cong G\actson (X,\mu^{*})$ given by taking inverses.

Thus, if we let $\mathcal{S}=\{Y\in \Sub(X):m_{Y}\in \mathcal{P}\},$ then by Lemma \ref{L:maximality argument} (\ref{I:uniqueness semigroup containment}), we may find a largest element $Y\in \mathcal{S}.$ By construction $G\actson (Y,m_{Y})$ is weakly contained in $\mathcal{C}$, and is the largest closed subgroup of $X$ with this property.

(\ref{I:minimal property}): First, suppose that $G\actson (Z,\zeta)$ is weakly contained in $\mathcal{C}$ and that $\psi\colon Z\to X$ is Borel  and $G$-equivariant. Then $\mu=\psi_{*}(\zeta)\in \mathcal{P}$ and so by Lemma \ref{L:maximality argument} (\ref{I:coset support}), there must be an $a\in X$ so that $\supp(\mu)\subseteq aY.$ Thus $\psi(z)\in aY$ for almost every $z\in Z.$

Conversely, suppose that $\widetilde{Y}$ is another closed, $G$-invariant subgroup of $X$ so that every action weakly contained in $\mathcal{C}$ is $G\actson X/\widetilde{Y}$-ergodic in the presence of $G\actson X.$ Since $G\actson (Y,m_{Y})$ is weakly contained in $\mathcal{C},$ we may apply the assumptions on $\widetilde{Y}$ to the inclusion map $\iota\colon Y\to X$ to see that $m_{Y}(Y\cap a\widetilde{Y})=1$ for some $a\in X.$ But this clearly implies that $Y\cap a\widetilde{Y}$ is dense in $Y$ and thus, since $\widetilde{Y}$ is closed, that $Y\subseteq a\widetilde{Y}.$ Since $1\in Y,$ we must have that $1\in a\widetilde{Y}.$ Since cosets of $\widetilde{Y}$ are either equal or disjoint, we must have that $a\widetilde{Y}=\widetilde{Y}.$ Thus we have shown that $Y\subseteq \widetilde{Y}.$ So  $Y$ is the smallest closed, $G$-invariant subgroup of $X$ so that every element weakly contained in $\mathcal{C}$ is $G\actson X/Y$-ergodic in the presence of $G\actson X.$
%

\end{proof}

Applying $\mathcal{C}$ to the class of Bernoulli shifts shows the following.

\begin{cor}\label{C:Bern weak contain}
Let $X$ be a compact, metrizable group and let $G$  be a countable, discrete group with $G\actson X$ by continuous automorphisms. The following are equivalent:
\begin{enumerate}[(i)]
\item the action $G\actson (X,m_{X})$ is weakly contained in a Bernoulli shift. \label{I:WC BErn}
\item For every proper $G$-invariant closed subgroup $Y$ of $X,$ there is an action weakly contained in a Bernoulli shift which is not $G\actson X/Y$-ergodic in the presence of $G\actson X.$

\label{I:ergodic in the presence}
\end{enumerate}
\end{cor}



The following two corollaries are simple consequences of Corollary \ref{C:structure corollary}.

\begin{cor}
Let  $G$ be a countable, discrete, sofic group, and let $G\actson X$ be an algebraic action. Then:
\begin{enumerate}[(i)]
    \item there is a largest closed, $G$-invariant $Y\leq X$ so that $G\actson (Y,m_{Y})$ is sofic with respect to $(\sigma_{k})_{k},$
\item $Y$ is the smallest closed, $G$-invariant subgroup of $X$ so that every sofic action is $G\actson X/Y$-ergodic in the presence of $G\actson X.$
\end{enumerate}
\end{cor}

\begin{cor}\label{C:strongly sofic}
Let  $G$ be a countable, discrete, sofic group, and let $G\actson X$ be an algebraic action. Fix a sofic approximation $(\sigma_{k})_{k}$ of $G.$ Then:
\begin{enumerate}[(i)]
    \item there is a largest $G$-invariant $Y\leq X$ so that $G\actson (Y,m_{Y})$ is strongly sofic.
    \item $Y$ is the smallest closed subgroup of $X$ so that every probability measure-preserving action which is strongly sofic with respect to $(\sigma_{k})_{k}$ is $G\actson X/Y$-ergodic in the presence of $G\actson X$.
\end{enumerate}
\end{cor}

For later use, we state a consequence of Corollary \ref{C:strongly sofic} for the study of algebraic actions with completely positive entropy. We will actually study completely positive entropy in the presence. See the discussion preceding Proposition 2.12 in \cite{Me13} for a precise definition of completely positive entropy in the presence. 

Though we  will not need the precise definition of entropy in the presence, we will briefly describe what entropy in the presence is  as well of the history of its definition. Suppose that $G$ is sofic with sofic approximation $\sigma_{k}\colon G\to S_{d_{k}}$ (where $S_{d_{k}}$ is the group of permutations on $d_{k}$ letters). Suppose that $G\actson (X,\mu),$ $G\actson (Y,\nu)$ are measure-preserving actions of $G$ on Lebesgue spaces, and that $G\actson (Y,\nu)$ is a factor of $G\actson (X,\mu)$ with factor map $\pi.$ We then have the notion of \emph{the entropy of $G\actson (Y,\nu)$ in the presence of $G\actson (X,\mu)$,} denoted $h_{(\sigma_{k})_{k}}(G\actson (Y,\nu):(X,\mu)).$ The term ``in the presence" here is borrowed from Voiculescu's notion of free entropy in the presence first defined in \cite{FreeEntropyDimensionIII}. This notion was implicitly defined by Kerr in \cite{KerrPartition}, and explicitly in \cite[Definition 2.7]{Me7} where a definition was given in terms of a given compact model for $G\actson (X,\mu),G\actson (Y,\nu)$ (in \cite{Me7} this is denoted by $h_{(\sigma_{k})_{k},\mu}(Y:X,G)).$ By \cite[Theorem 2.10]{Me7} this version in terms of a given compact model agrees with the version defined implicitly by Kerr in \cite{KerrPartition}. Theorem 2.10 of \cite{Me7} is intuitively obvious: entropy in the presence measures ``how many microstates for $G\actson (Y,\nu)$ have lifts to microstates for $G\actson (X,\mu).$" Kerr's version formulates this notion of ``how many microstates lift" using partitions, and the work in \cite{Me7} formulates this notion using a compact model. But both of these are measurements of how many microstates for $G\actson Y$ lift, and by methods now entirely standard in the field (first appearing in \cite{KLi2}) it is easy to equate the two quantities.  Shortly after the work in \cite{Me7}, a definition of topological entropy in the presence was given in \cite[Definition 9.3]{LiLiang2}. In \cite{LiLiang2}, the authors call this ``the entropy of $G\actson (Y,\nu)$ relative to the extension $G\actson (X,\mu).$" We prefer the name ``entropy in the presence" to avoid confusion with the ``entropy of $G\actson (X,\mu)$ relative to $G\actson (Y,\nu)$" defined in \cite{Me13}, but this is just a matter of taste. For example, if $G$ is amenable then the entropy of $G\actson (Y,\nu)$ in the presence of $G\actson (X,\mu)$ is just the entropy of $G\actson (Y,\nu)$ (by \cite[Theorem A.2]{Me7}), which is clearly not equal to the entropy of $G\actson (X,\mu)$ relative to $G\actson (Y,\nu).$ For example, take $X=Y.$ Then the entropy of $G\actson (X,\mu)$ relative to the extension $G\actson (X,\mu)$  as defined in \cite{LiLiang2} is equal to the entropy of $G\actson (X,\mu)$ (even in the sofic case), whereas the usual entropy of $G\actson (X,\mu)$ relative to $G\actson (X,\mu)$ is zero. Similar comments hold for the topological case. See \cite[Definition 3.3, Definition 3.4]{Me12} for a comparison of the definitions of topological and measure entropy in the presence. See the comments after Theorem 2.4 of \cite{SewardKrieger2} for the related notion of \emph{outer Rokhlin entropy} which is analogous to entropy in the presence for Rokhlin entropy (see e.g. \cite[Proposition 2.13]{Me13} for a comparison of the two quantities).

The notion of entropy in the presence is designed to fix the fact that sofic entropy as defined by Bowen can increase under factor maps. For example, it is a trivial consequence of the definitions that $h_{(\sigma)_{k}}(G\actson (Y,\nu):(X,\mu))\leq h_{(\sigma_{k})_{k}}(G\actson (X,\mu)).$ It is also straightforward to see that $h_{(\sigma_{k})_{k}}(G\actson (Y,\nu):(X,\mu))$ is increasing if we replace $G\actson(Y,\nu)$ with an intermediate factor between $Y$ and $X,$ and decreasing if we replace $G\actson (X,\mu)$ with an extension. Because it fixes the fact that entropy can increase under factors, entropy in the presence often gives the correct way to define properties of actions which depend upon their factors. For instance, it gives a different notion of a Pinsker factor (called the outer Pinsker factor) which has better properties than the usual Pinsker factor (see e.g. \cite{Me7,Me13}). It also gives a different notion of complete positive entropy. We say that $G\actson (X,\mu)$ has \emph{completely positive entropy in the presence}, if for every nontrivial factor $G\actson (Y,\nu)$ of $G\actson (X,\mu)$ the entropy of $G\actson (Y,\nu)$ in the presence of $G\actson (X,\mu)$ is positive. This trivially implies that every nontrivial factor of $G\actson (X,\mu)$ has positive entropy.


\begin{cor}\label{C:generating  by cpe subgroups}
Let $G$ be a countable, discrete, sofic group with sofic approximation $(\sigma_{k})_{k}.$ Let $G\actson X$ be an algebraic action. Suppose that $(X_{j})_{j\in J}$ are $G$-invariant, closed subgroups of $X,$ and that
\[X=\bigvee_{j\in J}X_{j}.\]
If each $G\actson (X_{j},m_{X_{j}})$ is strongly sofic and has completely positive measure entropy in the presence with respect to $(\sigma_{k})_{k}$, then $G\actson (X,m_{X})$ is strongly sofic and has completely positive measure entropy in the presence with respect to $(\sigma_{k})_{k}.$
\end{cor}

\begin{proof}
The fact that $G\actson (X,m_{X})$ is strongly sofic with respect to $(\sigma_{k})_{k}$ is automatic from Corollary \ref{C:strongly sofic}, so we turn to proving that $G\actson (X,m_{X})$ has completely positive measure entropy in the presence. Since $G\actson (X,m_{X})$ is strongly sofic, it follows from \cite[Theorem 1.3]{Me13} that there is a closed, normal subgroup $Y$ of $X$ so that the Pinsker factor is of the form $G\acston X/Y$ (and the factor map is just the quotient map $q_{Y}\colon X\to X/Y$).

For  $j\in J,$ let $K_{j}=q_{Y}(X_{j}).$ By \cite[Theorem 1.1]{Me12} and the fact that $G\actson (X_{j},m_{X_{j}})$ is strongly sofic with respect to $(\sigma_{k})_{k}$, we then have that
\begin{align*}
    h_{(\sigma_{k})_{k}}(G\actson (K_{j},m_{K_{j}}):(X_{j},m_{X_{j}}))=h_{(\sigma_{k})_{k},\topo}(G\actson K_{j}:X_{j})&\leq h_{(\sigma_{k})_{k},\topo}(G\actson X/Y:X)\\
    &=h_{(\sigma_{k})_{k}}(G\actson (X/Y,m_{X/Y}):(X,m_{X})),
\end{align*}
where in the last line we use  that $G\actson (X,m_{X})$ is strongly sofic.
By definition of the outer Pinsker factor,
\[h_{(\sigma_{k})_{k}}(G\actson (K_{j},m_{X_{j}}):(X_{j},m_{X_{j}}))\leq 0.\]
Since $G\actson (X_{j},m_{X_{j}})$ has completely positive entropy in the presence, the above shows that  $K_{j}=\{1\}.$ Since $q_{Y}$ is continuous and $X=\bigvee_{j}X_{J},$ it follows that
\[X/Y=q_{Y}(X)=\bigvee_{j}q_{Y}(X_{j})=\{1\}.\]
By definition of the outer Pinsker factor, this shows that $G\actson (X,m_{X})$ has completely positive entropy in the presence.

\end{proof}

%
%
%

\section{Measurably Extending Convolution}\label{S:extend convolve}

In this section, we provide the necessary background material for our main applications of Corollary \ref{C:Bern weak contain}. The main tool ends up being a measurable way to extend the convolution operation between real-valued functions.

 If $V$ is a vector space over $\R,$ and $m,k\in \N\cup\{\infty\}$ we will regard $M_{m,k}(V)$ as the vector space of all $m\times k$-matrices with entries in $V$. Right now this is a formal object without additional algebra structure, but of course it will have such a structure once $V$ is an algebra. It should be clear what $M_{m,k}(V)$ means if $m,k$ are finite. If, say, $m$ is infinite and $k$ is finite, we simply mean all doubly indexed arrays $(v_{ij})_{1\leq j\leq k,i\in \N}$ with $v_{ij}\in V$ for all $i,j.$ Similar remarks apply to the case that $k$ is infinite and $m$ is not, and when both are infinite. We identify $V^{m}$ with $M_{m,1}(V).$

Let $\C(G)$ denote the ring which is $c_{c}(G,\C)$ as a set and with the product operation of convolution, i.e.
\[(\alpha\beta)(g)=\sum_{h\in G}\alpha(h)\beta(h^{-1}g)\mbox{ for $\alpha,\beta\in \C(G).$}\]
Let $\R(G),\Z(G)$ etc. denote the subrings corresponding to $c_{c}(G,\R),c_{c}(G,\Z).$
For a $m\in \N\cup\{\infty\},$ we identify $\C(G)^{m}$ with $c_{c}(G,\C^{m}),$ and $\C(G)^{\oplus m}$ with the subset of $\C(G)^{m}$ consisting of those $\alpha$ so that the function $\{j\in \N:j\leq m\}\times G\to \C$ given by $(j,g)\mapsto \alpha(g)(j)$ is finitely supported. Similar remarks apply to $\R(G)^{\oplus m},\R(G)^{m},\Z(G)^{\oplus m},\Z(G)^{m}.$

Now let $m\in \N\cup\{\infty\},k\in \N.$ If $\xi\in M_{k,m}(\C(G)),$ then we define a linear map $r(\xi)\colon (\C^{k})^{G}\to (\C^{m})^{G}$ by
\[(r(\xi)\zeta)(g)(j)=\sum_{l=1}^{k}\sum_{h\in G}\zeta(h)(l)\xi_{lj}(h^{-1}g)\mbox{ for $\zeta\in (\C^{k})^{G},g\in G,1\leq j\leq m$.}\]
In the above equation, we are making the following convention: when $m\in \N\cup \{\infty\},$ to say that some equation (or condition) holds for all $1\leq j\leq m$ means that it holds for all \emph{integers} $j$ with $1\leq j\leq m.$ So in the above equation, when $m=\infty,$ we are only defining $(r(\xi)\zeta)(g)(j)$ for $j\in \N.$ We keep this convention throughout the rest of the paper.
This expression defining $r(\xi)\zeta$ makes sense, since for every $1\leq l\leq k,1\leq j\leq m$ we have that $\xi_{lj}$ is compactly supported. For later use, we also note that the same equation as above allows us to define $r(\xi)\zeta$ for $\xi\in M_{k,m}(\C^{G}),\zeta\in \C(G)^{\oplus k}.$

If $m\in \N,k\in \N\cup\{\infty\},$ then for $\xi\in M_{k,m}(\C(G))$ we also have a linear map $\lambda(\xi)\colon (\C^{m})^{G}\to (\C^{k})^{G}$ by
\[(\lambda(\xi)\zeta)(g)(j)=\sum_{l=1}^{m}\sum_{h\in G}\xi_{jl}(h)\zeta(h^{-1}g)(l),\mbox{ for $\zeta\in (\C^{m})^{G},g\in G,1\leq j\leq k.$}\]
Similar remarks as above allow us to define $\lambda(\xi)\zeta$ for $\xi\in M_{k,m}(\C^{G}),\alpha\in \C(G)^{\oplus m}.$
For $g\in G,A\in M_{k,m}(\C^{G}),$ we let $\lambda(g)A$ be given by $(\lambda(g)A)_{ij}=\lambda(g)(A_{ij}).$ Technically, we are in some sense multiplying $A$ by the matrix $B\in M_{k}(\C^{G})$ with $B_{ij}=\delta_{i=j}g,$ and we should adopt notation to account for this. We think this mild abuse of notation will not cause any problems.

 For $\xi\in \C^{G},$ we let $\xi^{*}\in \C^{G}$ be given by $\xi^{*}(g)=\overline{\xi(g^{-1})}.$
Notice that if $\alpha,\beta\in \C(G),$ then $(\alpha\beta)^{*}=\beta^{*}\alpha^{*}.$
For $m,k\in \N\cup\{\infty\},$ $A\in M_{m,k}(\C^{G}),$ we define $A^{*}\in M_{k,m}(\C^{G})$ by
\[(A^{*})_{ij}=(A_{ji})^{*}.\]

For $m\in \N\cup\{\infty\},$ we let $q\colon (\R^{m})^{G}\to (\T^{m})^{G}$ be given by
\[q(x)(g)(j)=x(g)(j)+\Z\mbox{ for $g\in G,1\leq j\leq m.$}\]
For the remainder of the paper, we will reserve $q$ for the above quotient map. We will suppress the dependence upon $m,G$ in the notation.
For $\xi\in M_{m,k}(\R(G))$ we define maps
\[\Theta_{\xi}\colon (\R^{k})^{G}\to (\T^{m})^{G},\Psi_{\xi}\colon (\R^{k})^{G}\to \R^{m}\]
by
\[\Theta_{\xi}(x)=q(r(\xi^{*})x),\Psi_{\xi}(x)=(r(\xi^{*})x)(1).\]

Clearly, $\Theta_{\xi}$ is defined via convolution and this is simple to make sense of when $\xi_{ij}$ is compactly supported for each $i,j.$ The major goal of this section is to extend this to a larger vector space of $\xi,$ and we will do this via a continuity argument. So we will need topologies on $M_{m,k}(\R(G))$ and the space of maps $(\R^{k})^{G}\to (\T^{m})^{G}.$

We put a topology on $M_{m,k}(\R(G))$ by embedding it in $M_{m,k}(\ell^{2}(G,\R)).$ We give $M_{m,k}(\ell^{2}(G,\R))$ the product topology, i.e. a sequence $(\xi^{(n)})_{n}$ in $M_{m,k}(\ell^{2}(G,\R))$ converges to $\xi \in M_{m,k}(\ell^{2}(G,\R)$ if
\[\|\xi^{(n)}_{ij}-\xi_{ij}\|_{2}\to_{n\to\infty}0\mbox{ for all $1\leq i\leq m,1\leq j\leq k.$}\]
If both $m,k$ are finite, this can be given by the $\|\cdot\|_{2}$-norm on $M_{m,k}(\ell^{2}(G,\R)):$
\[\|\xi\|_{2}^{2}=\sum_{i,j}\|\xi_{ij}\|_{2}^{2},\mbox{\,\,\, for $\xi\in M_{m,k}(\ell^{2}(G,\R)).$}\]
As for the space of maps $(\R^{k})^{G}\to (\T^{m})^{G},$ we use the following notion. Let $X$ be a standard Borel space, and $\mu$ a completed Borel probability measure on $X.$ Given a metric space $(Y,d)$ we let $\Meas(X,\mu,Y)$ be the space of $\mu$-measurable maps $f\colon X\to Y.$ As is typical in measure theory we will identify two such maps if they agree almost everywhere, but will almost always suppress this from the notation. We equip $\Meas(X,\mu,Y)$ with the metric
\[d_{m}(f,g)=\int_{X}\min(d(f(x),g(x)),1)\,d\mu(x).\]
It is well known that this gives $\Meas(X,\mu,Y)$ the measure topology: a sequence $(f_{n})_{n}$ in $\Meas(X,\mu,Y)$ converges with respect to $d_{m}$ to an $f\in \Meas(X,\mu,Y)$ if and only if it converges in measure, i.e.  for every $\varepsilon>0,$
\[\mu(\{x:d(f_{n}(x),f(x))>\varepsilon\})\to_{n\to\infty}0.\]
We will not work explicitly with this metric. All we will need to use is that if $(Z,\Delta)$ is another metric space, then $\Phi\colon Z\to \Meas(X,\mu,Y)$ is uniformly continuous if and only if for every $\varepsilon>0,$ there is a $\delta>0$ so that if $z_{1},z_{2}\in Z$ and $\Delta(z_{1},z_{2})<\delta,$ then $\mu(\{x:d(\Phi(z_{1})(x),\Phi(z_{2})(x))>\varepsilon\})<\varepsilon.$
This is a well known exercise that we leave to the reader.
If $K$ is a compact, metrizable space, then the notion of uniform continuity for functions defined on $\Meas(X,\mu,K)$ does not depend on a choice of a compatible metric on $K,$ so we will typically not explicitly put a metric on $K.$

Unless otherwise specified, for a $k\in \N$ we will endow $\R^{k}$ with the metric induced from the $\|\cdot\|_{2}$-norm:
\[\|x\|_{2}=\left(\sum_{j}|x_{j}|^{2}\right)^{1/2}.\]
We endow $\R^{\infty}$ with the metric
\[d(x,y)=\left(\sum_{n=1}^{\infty}2^{-n}\min(1,|x(n)-y(n)|^{2})\right)^{1/2}.\]
For $k\in \N$ we also endow $M_{\infty,k}(\ell^{2}(G,\R))$ with the metric
\[d(A,B)=\left(\sum_{n=1}^{\infty}\sum_{l=1}^{k}2^{-n}\min(1,\|A_{n,l}-B_{n,l}\|_{2}^{2})\right)^{1/2}.\]
However, we will rarely use these metrics. We will only use the following facts (which are standard metric space exercises):
\begin{itemize}
    \item if $(\Omega,\rho)$ is a metric space, then $f\colon \Omega\to \R^{\infty}$ is uniformly continuous if and only if $\pi_{j}\circ f$ is uniformly continuous for all $j\in\N,$ where $\pi_{j}\colon \R^{\infty}\to \R$ is given by $\pi_{j}(x)=x(j).$
    \item The maps $\pi_{j}\colon \R^{\infty}\to \R$ for $j\in \N$ are uniformly continuous.
    \item Define maps $ \Pi_{l}\colon M_{\infty,k}(\ell^{2}(G,\R))\to M_{1,k}(\ell^{2}(G,\R))$  by
    \[\Pi_{l}\left(\begin{bmatrix}
    \xi_{1}\\
    \xi_{2}\\
    \vdots
    \end{bmatrix}\right)=\xi_{l}.\] If $(\Omega,\rho)$ is a metric space, then $f\colon \Omega\to M_{\infty,k}(\ell^{2}(G,\R))$ is uniformly continuous if and only if $\Pi_{l}\circ f\colon \Omega\to M_{1,k}(\ell^{2}(G,\R))$ is uniformly continuous for all $l\in \N,$
    \item the maps $\Pi_{l}\colon M_{\infty,k}(\ell^{2}(G,\R))\to M_{1,k}(\ell^{2}(G,\R))$
    are all uniformly continuous for $l\in \N.$
\end{itemize}

We say that $\nu\in \Prob(\R^{k})$ has \emph{a finite second moment} if $\int \|x\|_{2}^{2}\,d\nu(x)<\infty.$ We say that $\nu$ has \emph{mean zero} if $\int \|x\|_{2}\,d\nu(x)<\infty$ and $\int x(j)\,d\nu(x)=0$ for all $j=1,\cdots,k.$

Finally, given a Lebesgue space $(X,\mu)$ we let $L^{2}_{0}(X,\mu)$ be the subspace of $L^{2}(X,\mu)$ consisting of functions $f$ with $\int f\,d\mu=0.$

\begin{prop}\label{P:hs nonsense}
Let $G$ be a countable, discrete, group and let $k\in \N,m\in \N\cup\{\infty\}.$ Suppose that $\nu\in \Prob(\R^{k})$ has mean zero and a finite second moment. Then the map $\xi\in M_{m,k}(\R(G))\mapsto \Psi_{\xi}$ extends uniquely to a uniformly continuous map $M_{m,k}(\ell^{2}(G,\R))\to \Meas((\R^{k})^{G},\nu^{\otimes G},\R^{m}).$

\end{prop}

\begin{proof}
For $1\leq l\leq m,$ let $\pi_{l}\colon \R^{m}\to \R$ be given by $\pi_{l}(x)=x(l).$ It suffices to show that for each $1\leq l\leq m,$ the map $\xi\mapsto \pi_{l}\circ \Psi_{\xi}$ has a uniformly continuous extension  $M_{m,k}(\ell^{2}(G,\R))\to \Meas((\R^{k})^{G},\nu^{\otimes G},\R).$

For $\xi\in M_{m,k}(c_{c}(G,\R))$ write
\[\xi=\begin{bmatrix}
\xi_{1}\\
\xi_{2}\\
\vdots
\end{bmatrix}.\]
Then $\pi_{l}\circ \Psi_{\xi}=\Psi_{\xi_{l}}$ for all $1\leq l\leq m.$ So it suffices to show that the map $\xi\mapsto \Psi_{\xi_{l}}$ has a continuous extension. Since $\xi\mapsto \xi_{l}$ is clearly uniformly continuous we may, and will, assume that $m=1.$

So let $\xi\in M_{1,k}(c_{c}(G,\R)).$ Write
\[\xi=\begin{bmatrix}
\xi_{1}&\xi_{2}&\cdots \xi_{k}
\end{bmatrix},\] where $\xi_{j}\in c_{c}(G,\R)$ for $1\leq j\leq k.$ For $1\leq j\leq k,$ let $X_{j}\in L^{2}_{0}((\R^{k})^{G},\nu^{\otimes G})$ be given by $X_{j}(x)=x(1)(j).$ Observe that $X_{j}$ is indeed in $L^{2}_{0}((\R^{k})^{G},\nu^{\otimes G}),$ as $\nu$ has mean zero and a finite second moment. For $g\in G,$ let $\rho(g)$ be the right shift action of $G$ on $(\R^{k})^{G}$ and consider $\rho$ as unitary representation $\rho\colon G\to \mathcal{U}(L^{2}((\R^{k})^{G},\nu^{\otimes G}))$
defined by $(\rho(g)f)(x)=f(\rho(g)^{-1}x).$
Then, for every $\alpha\in c_{c}(G,\C)$ we may define $\rho(\alpha)\in B(L^{2}((\R^{k})^{G},\nu^{\otimes G}))$ by
\[\rho(\alpha)=\sum_{g\in G}\alpha(g)\rho(g).\]
Then $\rho$ is a $*$-representation of the $*$-algebra $\C(G).$ Moreover, it is easy to see that
\[\Psi_{\xi}=\sum_{j=1}^{k}\rho(\xi_{j}^{*})X_{j}.\]
For $1\leq j\leq k,$ we have that
\[\|\rho(\xi_{j}^{*})X_{j}\|_{2}^{2}=\ip{\rho(\xi_{j}^{*})X_{j},\rho(\xi_{j}^{*})X_{j}}=\ip{\rho(\xi_{j}*\xi_{j}^{*})X_{j},X_{j}}.\]
Since $X_{j}$ has  mean zero and independent translates, the above equality implies that
\[\|\rho(\xi_{j}^{*})X_{j}\|_{2}^{2}=(\xi_{j}*\xi_{j}^{*})(1)\|X_{j}\|_{2}^{2}=\|\xi_{j}\|_{2}^{2}\|X_{j}\|_{2}^{2}.\]
So by the Cauchy-Schwartz inequality,
\[\|\Psi_{\xi}\|_{2}\leq \sum_{j=1}^{k}\|\xi_{j}\|_{2}\|X_{j}\|_{2}\leq  \|\xi\|_{2}\left(\sum_{j}\|X_{j}\|_{2}^{2}\right)^{1/2}.\]
Since the map $M_{1,k}(\R(G))\to L^{2}((\R^{k})^{G},\nu^{\otimes G},\R)$ given by $\xi\mapsto \Psi_{\xi}$ is linear, the above shows that it is Lipschitz. So it has a unique uniformly continuous extension to a map $M_{1,k}(\ell^{2}(G,\R))\to L^{2}((\R^{k})^{G},\nu^{\otimes G},\R).$ Since the inclusion
\[L^{2}((\R^{k})^{G},\nu^{\otimes G},\R)\hookrightarrow \Meas((\R^{k})^{G},\nu^{\otimes G},\R)\]
 is uniformly continuous, we are done.

\end{proof}

\begin{cor}\label{C:meas extens}
Let $G$ be a countable discrete group, $k\in \N,m\in \N\cup\{\infty\}$ and $\nu\in \Prob(\R^{k}).$ We may  uniquely extend the map $M_{m,k}(\R(G))\to \Meas((\R^{k})^{G},\nu^{\otimes G}, (\T^{G})^{m}),$ $\xi\mapsto \Theta_{\xi}$ to a map $M_{m,k}(\ell^{2}(G,\R))\to\Meas((\R^{k})^{G},\nu^{\otimes G}, (\T^{m})^{G}) $ which is  uniformly continuous. Furthermore, if we continue to denote $\Theta_{\xi}$ the image of $\xi\in M_{m,k}(\ell^{2}(G,\R))$ under this map, then $\xi\mapsto \Theta_{\xi}$ satisfies the following properties:
\begin{enumerate}[(a)]
\item \label{I:left shift equiv} $\Theta_{\xi}$ is equivariant with respect to the left shift actions of $G$ on $(\R^{k})^{G},(\T^{m})^{G}$ for every $\xi\in M_{m,k}(\ell^{2}(G,\R)),$ 
\item $\Theta_{\xi}\circ \rho(g)=\Theta_{\rho(g^{-1})\xi},$ for all $g\in G,\xi\in M_{m,k}(\ell^{2}(G,\R)),$ \label{I:right shift equiv}
\item $\xi\mapsto \Theta_{\xi}$ is an additive homomorphism. \label{I:additive homom}
\end{enumerate}
\end{cor}

\begin{proof}
By Proposition \ref{P:hs nonsense}, we may extend $\xi\mapsto \Psi_{\xi}$ to a uniformly continuous map $M_{m,k}(\ell^{2}(G,\R))\to \Meas((\R^{k})^{G},\nu^{\otimes G},\R^{m}).$ For $\xi\in M_{m,k}(\ell^{2}(G,\R))$ we continue to use $\Psi_{\xi}$ for the image of $\xi$ under this map. Given $\xi\in M_{m,k}(\ell^{2}(G,\R))$ we then define $\Theta_{\xi}\in\Meas((\R^{k})^{G},\nu^{\otimes G},(\T^{m})^{G})$ by
\[(\Theta_{\xi})(x)(g)=\Psi_{\xi}(g^{-1}x)+\Z^{m}.\]
It is direct to check if $\xi\in M_{m,k}(\R(G))$ then this agrees with the previous definition of $\Theta_{\xi}.$ To show that $\xi\mapsto \Theta_{\xi}$ is uniformly continuous, it suffices to show that for every $g\in G,$ $\xi\mapsto \mathcal{E}_{g}\circ \Theta_{\xi}$ is uniformly continuous where $\mathcal{E}_{g}\colon (\T^{m})^{G}\to \T^{m}$ is given by $\mathcal{E}_{g}(\theta)=\theta(g).$ Since $\mathcal{E}_{g}\circ \Theta_{\xi}=q\circ \Psi_{\xi}\circ g^{-1},$ the uniform continuity of the map $\xi\mapsto \Psi_{\xi}$ clearly implies uniform continuity of the map $\xi\mapsto \mathcal{E}_{g}\circ \Theta_{\xi}.$ This proves existence of the uniformly continuous extension of $\xi\mapsto \Theta_{\xi}.$ The uniqueness is clear by density of $M_{m,k}(\R(G))$ in $M_{m,k}(\ell^{2}(G,\R)).$

Items (\ref{I:left shift equiv}), (\ref{I:right shift equiv}) are easy in the case that $\xi\in M_{m,k}(\R(G))$ the general case follows by continuity of $\xi\mapsto \Theta_{\xi}.$ The same method applies to showing (\ref{I:additive homom}).

\end{proof}

We will need to control the image of $\Theta_{\xi}.$ For this purpose, we introduce a natural subgroup  of $(\T^{m})^{G}$ for every $\zeta\in M_{k,m}(\R^{G}).$

\begin{defn}\label{D:subgroup associated to matrix}
Let $m\in \N\cup\{\infty\},k\in \N,$ and $\zeta\in M_{k,m}(\R^{G}).$ For $1\leq j\leq k,$ we define $\zeta_{j}\in (\R^{m})^{G}$ by
\[\zeta_{j}(g)(l)=\zeta_{jl}(g) \mbox{ for all $1\leq l\leq m,g\in G.$}\]
We let $X^{\zeta}$ be the smallest, closed, $G$-invariant subgroup of $(\T^{m})^{G}$ containing all $q(\zeta_{j})$ for $1\leq j\leq k.$
\end{defn}

The following lemma will be crucial in our control of both the image of $\Theta_{\xi},$ as well as to gain information about $(\Theta_{\xi})_{*}(\nu^{\otimes G})$ . It will follow by similar arguments as in \cite[Lemma 3.1]{MeWE}.

\begin{lem}\label{L:L2 product computation}
Let $k\in \N$ and let $f\colon \R^{k}\to \C$ which can be written as $f(t)=1+\ip{F(t)t,t}$ where $F\colon \R^{k}\to M_{k}(\C)$ is continuous. Let $J$ be a countable set. Then for every $\xi\in \ell^{2}(J,\R)^{\oplus k}$ the product
\[\prod_{j\in J}f(\xi(j))\]
converges absolutely, and the map
$\Phi\colon \ell^{2}(J,\R)^{\oplus k}\to \C$
given by
\[\Phi(\xi)=\prod_{j\in J}f(\xi(j))\]
is continuous if we give $\ell^{2}(J,\R)^{\oplus k}$ the $\|\cdot\|_{2}$-topology.

\end{lem}

\begin{proof}
We first show that the product defining $\Phi$ converges absolutely. Fix $\xi\in \ell^{2}(J,\R)^{\oplus k}.$ Then $C=\sup_{j\in J}\|F(\xi(j))\|<\infty,$ so:
\[\sum_{j\in J}|1-f(\xi(j))|=\sum_{j\in J}|\ip{F(\xi(j))\xi(j),\xi(j)}|\leq C\sum_{j\in J}\|\xi(j)\|_{2}^{2}<\infty.\]
By a well known criterion, this implies that the product defining $\Phi$ converges absolutely.

It only requires a mildly more sophisticated argument to show that $\Phi$ is continuous. Fix a $\xi\in \ell^{2}(J,\R)^{\oplus k},$ and let $\xi_{n}$ be a sequence of vectors with $\|\xi-\xi_{n}\|_{2}\to 0.$ Let $\log$ be the analytic branch of the logarithm defined in $\C\setminus (-\infty,0]$ which has $\log(1)=0.$ We may choose a constant $A>0$ so that
\[|\log(z)-\log(w)|\leq A|z-w|\]
if $|z-1|,|w-1|<1/2.$ We may also choose a $\delta\in (0,1)$ so that $|f(t)-1|<1/2$ if $\|t\|_{2}<\delta.$ Let
$B=\sup_{\|t\|_{2}<\delta}\|F(t)\|.$
Let
\[E=\left\{j\in J:\|\xi(j)\|_{2}\geq \frac{\delta}{2}\right\}.\]
Choose $N$ large enough so that $\|\xi-\xi_{n}\|_{2}<\delta/2$ for all $n\geq N.$ Fix an $n\geq N.$ Then for any $j\in J\setminus E,$ we have:
\begin{align*}
    |\log(f(\xi(j)))-\log(f(\xi_{n}(j)))|\leq A|f(\xi_{j})-f(\xi_{n}(j))|&=A|\ip{F(\xi(j))\xi(j),\xi(j)}-\ip{F(\xi_{n}(j))\xi_{n}(j),\xi_{n}(j)}|\\
    &\leq A\|\xi(j)\|_{2}^{2}\|F(\xi_{n}(j))-F(\xi(j))\|\\
    &+A\left|\ip{F(\xi_{n}(j))\xi(j),\xi(j)}-\ip{F(\xi_{n}(j))\xi_{n}(j),\xi_{n}(j)}\right|.
\end{align*}
We have that
\begin{align*}
    \left|\ip{F(\xi_{n}(j))\xi(j),\xi(j)}-\ip{F(\xi_{n}(j))\xi_{n}(j),\xi_{n}(j)}\right|&\leq |\ip{F(\xi_{n}(j))(\xi(j)-\xi_{n}(j)),\xi(j)}|+|\ip{F(\xi_{n}(j))\xi_{n}(j),\xi_{n}(j)-\xi(j)}|\\
    &\leq B\|\xi(j)-\xi_{n}(j)\|_{2}(\|\xi(j)\|_{2}+\|\xi_{n}(j)\|_{2}).
\end{align*}
Hence for $n\geq N,$
\begin{align*}
    \sum_{j\in J\setminus E}|\log(f(\xi(j)))-\log(f(\xi_{n}(j)))|\leq AB\|\xi-\xi_{n}\|_{2}(\|\xi\|_{2}+\|\xi_{n}\|_{2})+A\sum_{j\in J\setminus E}\|\xi(j)\|_{2}^{2}\|F(\xi_{n}(j))-F(\xi(j))\|.
\end{align*}
The first term on the right-hand side of this inequality goes to zero as $n\to\infty,$ since $\|\xi-\xi_{n}\|_{2}\to 0.$ For the second term, observe that for $j\in J\setminus E$ we have that $\|F(\xi_{n}(j))-F(\xi(j))\|\to_{n\to\infty}0,$ and
$\|\xi(j)\|_{2}^{2}\|F(\xi_{n}(j))-F(\xi(j))\|\leq 2B\|\xi(j)\|_{2}^{2}$ for $n\geq N,j\in J\setminus E.$ Since $\xi\in \ell^{2}(J,\R)^{\oplus k},$ the dominated convergence theorem implies that
\[\sum_{j\in J\setminus E}\|\xi(j)\|_{2}^{2}\|F(\xi_{n}(j))-F(\xi(j))\|\to_{n\to\infty}0.\]
Hence,
    \[\lim_{n\to\infty}\sum_{j\in J\setminus E}\log f(\xi_{n}(j))=\sum_{j\in J\setminus E}\log f(\xi(j)).\]
Exponentiating,
\[\lim_{n\to\infty}\prod_{j\in J\setminus E}f(\xi_{n}(j))=\prod_{j\in J\setminus E}f(\xi(j)).\]
Since $E$ is finite,
\[\lim_{n\to\infty}\prod_{j\in E}f(\xi_{n}(j))=\prod_{j\in E}f(\xi(j)).\]
Thus we see that
\[\lim_{n\to\infty}\prod_{j\in J}f(\xi_{n}(j))=\prod_{j\in J}f(\xi(j)),\]
and this shows that $\Phi$ is continuous.
\end{proof}

To control the image of $\Theta_{\xi},$ we will also use the following fact. If $Y$ is a locally compact, abelian group, we use $\widehat{Y}$ for the group of continuous homomorphisms $\chi\colon Y\to \T.$ The group $\widehat{Y}$ is called the \emph{Pontryagin dual of $Y.$} If  $X\leq Y$ we let $X^{o}=\{\alpha\in \widehat{Y}:\alpha(x)=0\mbox{ for all $x\in X$}\}.$ It is a consequence of Pontryagin duality that
\[X=(X^{o})^{o},\]
(see \cite[Lemma 2.1.3]{Rud}).

\begin{prop}\label{P:annhilator computation}
Let $G$ be  countable discrete group, $k\in \N,m\in \N\cup\{\infty\}.$
For $\xi\in M_{k,m}(\ell^{2}(G,\R))$ we have that
\[(X^{\xi})^{o}=\{\alpha\in \Z(G)^{\oplus m}:r(\xi^{*})\alpha\in \Z(G)^{\oplus k}\}.\]
\end{prop}

\begin{proof}
Define $\xi_{1},\cdots,\xi_{k}$ as in Definition \ref{D:subgroup associated to matrix}.
First suppose that $\alpha\in \Z(G)^{\oplus m}$ and that $r(\xi^{*})\alpha\in \Z(G)^{\oplus k}.$ Then for every $1\leq j\leq k,g\in G$ we have, by a direct computation, that $\ip{\alpha,g\xi_{j}}=(r(\xi_{j}^{*})\alpha)(g).$ It is also easily seen that $(r(\xi^{*})\alpha)(g)(j)=(r(\xi^{*}_{j})\alpha)(g)$ for $g\in G.$ So,  we have for all  $1\leq j\leq k,g\in G$ that
\[\ip{\alpha,g\xi_{j}}=(r(\xi_{j}^{*})\alpha)(g)=(r(\xi^{*})\alpha)(g)(j)\in \Z.\]
So $\ip{\alpha,g\xi_{j}}\in \Z$ for all $1\leq j\leq k,g\in G.$ Since the group generated by $\{gq(\xi_{j}):g\in G,1\leq j\leq k\}$ is dense in $X^{\xi},$ this implies that $\alpha\in (X^{\xi})^{o}.$

Conversely, suppose that $\alpha\in (X^{\xi})^{o}.$ Then for every $1\leq j\leq k, g\in G$ we have that
\[(r(\xi^{*})\alpha)(g)(j)=(r(\xi_{j}^{*})\alpha)(g)=\ip{\alpha,g\xi_{j}}\in\Z.\]
Thus $r(\xi^{*})\alpha\in \Z(G)^{\oplus k}.$

\end{proof}

We now close with two more properties of the map $\xi\mapsto \Theta_{\xi}$ that will be crucial for us, one of which gives us good control over the image of $\Theta_{\xi}$ that we alluded to earlier. Given $\mu\in \Prob(Y),$ we define its \emph{Fourier transform} $\widehat{\mu}\colon \widehat{Y}\to \C$ by
\[\widehat{\mu}(\chi)=\int_{Y}\exp(2\pi i \chi(y))\,d\mu(y).\]

We identify $\widehat{\R}$ with $\R$ by the duality $\ip{t,s}_{\T}=ts+\Z.$ Given a countable, discrete group $G,$ and $m\in \N\cup\{\infty\}$ we identify  the Pontryagin dual of $(\T^{m})^{G}$ with $\Z(G)^{\oplus m}$ via the duality
\[\ip{\theta,\alpha}_{\T}=\sum_{l=1}^{m}\sum_{g\in G}\alpha(g)(l)\theta(g)(l).\]
Since $\alpha$ is finitely supported, this is in fact a finite sum.

\begin{thm}\label{T:pushforward props}
Let $G$ be a countable, discrete group, $k\in \N,m\in \N\cup\{\infty\}$ and $\nu\in \Prob(\R^{k}).$ Suppose that $\nu$ has mean zero and a finite second moment.
\begin{enumerate}[(i)]
\item If $\xi\in M_{m,k}(\ell^{2}(G,\R)),$ and we set $\mu_{\xi}=(\Theta_{\xi})_{*}(\nu^{\otimes G}),$ then \label{I:Fourier transform formula equiv}
\[\widehat{\mu}_{\xi}(\alpha)=\prod_{g\in G}\widehat{\nu}((r(\xi)\alpha)(g))\]
for all $\alpha\in \Z(G)^{\oplus m}.$
\item If $\xi\in M_{m,k}(\ell^{2}(G,\R)),$ and $\nu$ is supported on $\Z^{k},$ then $\Theta_{\xi}$ is almost surely valued in $X^{\xi^{*}}$. \label{I:controlling the image of conv}
\end{enumerate}
\end{thm}

\begin{proof}

(\ref{I:Fourier transform formula equiv}): By Corollary \ref{C:meas extens}, the map $M_{m,k}(\ell^{2}(G,\R))\to \C^{\Z(G)^{\oplus m}}$ given by $\xi\mapsto \widehat{\mu}_{\xi}$ is continuous. Since $\nu$ has mean zero and finite second moments, we know that $\widehat{\nu}(0)=1,(\nabla\widehat{\nu})(0)=0,$ and that $\widehat{\nu}$ is $C^{2}.$ Thus by Taylor's theorem with remainder, we know that $\widehat{\nu}(t)=1+\ip{F(t)t,t}$ for some continuous $F\colon \R^{k}\to M_{k}(\C).$ So by Lemma \ref{L:L2 product computation} we also have that the map $M_{m,k}(\ell^{2}(G,\R))\to \C^{\Z(G)^{\oplus m}}$ given by
\[\xi\mapsto \left(\alpha\mapsto \prod_{g\in G}\widehat{\nu}((r(\xi)\alpha)(g))\right)\]
is continuous.
By a direct computation, $\xi\mapsto \widehat{\mu}_{\xi}$ and
\[\xi\mapsto \left(\alpha\mapsto \prod_{g\in G\\ }\widehat{\nu}((r(\xi)\alpha)(g))\right)\]
agree on $M_{m,k}(c_{c}(G,\R)).$ As both these maps are continuous they must  be equal by density of $M_{m,k}(c_{c}(G,\R))$ inside of $M_{m,k}(\ell^{2}(G,\R)).$

(\ref{I:controlling the image of conv}): It suffices to show that $\mu_{\xi}$ is supported on $X^{\xi^{*}}.$ To show this, it suffices to show that if $\alpha\in (X^{\xi^{*}})^{o},$ then $\widehat{\mu}_{\xi}(\alpha)=1.$
If $\alpha\in (X^{\xi^{*}})^{o},$ then by Proposition \ref{P:annhilator computation} we have $r(\xi)\alpha\in \Z(G)^{\oplus k}.$ Since $\nu\in \Prob(\Z^{k}),$ we have that $\widehat{\nu}(\Z^{k})=\{1\}.$ So
\[\widehat{\mu}_{\xi}(\alpha)=\prod_{g\in G}\widehat{\nu}((r(\xi)\alpha)(g))=1.\]

\end{proof}

We now show that Corollary \ref{C:meas extens} is in some sense optimal, and that one cannot hope to extend $\Theta_{\xi}$ to the case of $\xi\in M_{m,k}(\ell^{p}(G,\R)).$
For $1\leq p\leq \infty,$ and $m,k\in \N\cup\{\infty\}$ we can give $M_{m,k}(\ell^{p}(G,\R))$ the product topology: a sequence $\xi^{(n)}\in M_{m,k}(\ell^{p}(G,\R))$ converges if and only if
\[\|\xi^{(n)}_{ij}-\xi_{ij}\|_{p}\to_{n\to\infty}0\mbox{ for all $1\leq i\leq m,1\leq j\leq k$}.\]
If both $m,k$ are finite, this topology is given by the norm
\[\|\xi\|_{p}=\left(\sum_{i,j}\|\xi_{ij}\|_{p}^{p}\right)^{1/p}\mbox{ for $1\leq p<\infty$}\]
\[\|\xi\|_{\infty}=\max_{i,j}\|\xi_{ij}\|_{\infty}.\]
Recall that we are identifying $M_{m,1}(\R^{G})$ with $(\R^{G})^{m}.$ So when $m$ is finite, this allows us to consider the norm $\|\cdot\|_{p}$ on $\ell^{p}(G,\R)^{\oplus m}.$

\begin{prop}\label{P:cant extend}
Let $G$ be a countably infinite group, and $k\in \N,m\in \N\cup\{\infty\}.$ Suppose that $\nu\in \Prob(\R^{k})$ and that $\nu$ is not the dirac mass at $0.$ Fix a $p\in (2,\infty).$
\begin{enumerate}[(i)]
\item There is no extension of the map  $M_{m,k}(\R(G))\to \Meas((\R^{k})^{G},\nu^{\otimes G},(\T^{m})^{G}),$ $\xi\mapsto \Theta_{\xi}$ to a continuous map $M_{m,k}(\ell^{p}(G,\R))\to \Meas((\R^{k})^{G},\nu^{\otimes G},(\T^{m})^{G}).$ \label{I:cant extend map}
\item There is no  extension of the map  $M_{m,k}(\R(G))\to \Prob((\T^{m})^{G}),$ $\xi\mapsto (\Theta_{\xi})_{*}(\nu^{\otimes G})$ to a continuous map $M_{m,k}(\ell^{p}(G,\R))\to \Prob((\T^{m})^{G}).$ \label{I:cant probability extend}
\end{enumerate}

\end{prop}

\begin{proof}
For $\xi \in M_{m,k}(c_{c}(G,\R)),$ we let $\mu_{\xi}=(\Theta_{\xi})_{*}(\nu^{\otimes G}).$ If we could continuously extend $\xi\mapsto \Theta_{\xi},$ then we could continuously extend $\xi\mapsto \mu_{\xi}$ by composing with the continuous map $\Meas((\R^{k})^{G},\nu^{\otimes G},(\T^{m})^{G})\to \Prob((\T^{m})^{G})$ given by $\Theta\mapsto \Theta_{*}(\nu^{\otimes G}).$ So it suffices to prove (\ref{I:cant probability extend}). As in the proof of Proposition \ref{P:hs nonsense}, we may, and will, assume that $m=1.$

We first claim the following.

\emph{Claim 1. There is a $1\leq j\leq k$ so that}
\[\lim_{t\to 0}\frac{|1-\widehat{\nu}(te_{j})|}{|t|^{p}}=\infty.\]
Suppose the claim is false. Fix a $1\leq j\leq k.$ By assumption, we can find a sequence $t_{n}\to 0$ with $t_{n}\ne 0$ for all $n$ so that
\[M=\sup_{n}\frac{|1-\widehat{\nu}(t_{n}e_{j})|}{|t_{n}|^{p}}<\infty.\]
By direct computation,
\[ \rea{\left(\frac{1-\widehat{\nu}(t_{n}e_{j})}{|t_{n}|^{2}}\right)}=2\int_{\R^{k}}\left(\frac{\sin(\pi t_{n}x_{j})}{t_{n}}\right)^{2}\,d\nu(x),\]
so
\[\int_{\R^{k}}\left(\frac{\sin(\pi t_{n}x_{j})}{t_{n}}\right)^{2}\,d\nu(x)\leq \frac{M}{2}|t_{n}|^{p-2}.\]
Applying Fatou's Lemma, it follows that
\[\int |x_{j}|^{2}\,d\nu(x)=0.\]
Since this is true for every $1\leq j\leq k,$ we have that
\[\int \|x\|_{2}^{2}\,d\nu(x)=0,\]
and this contradicts our assumption that $\nu$ is not the dirac mass at $0.$

Now suppose that a continuous extension of $\xi\mapsto \mu_{\xi}$ exists, we continue to denote the image of $\xi\in M_{1,k}(\ell^{p}(G,\R))$ under this extension by $\mu_{\xi}.$ We make another claim. For $\xi\in M_{1,k}(\ell^{p}(G,\R)),$ we define $\widetilde{\xi}\in \ell^{p}(G,\R^{k})$ by $\widetilde{\xi}(g)(j)=\xi_{1j}(g)$ for $1\leq j\leq k,g\in G.$

\emph{Claim 2. For every $\xi\in M_{1,k}(\ell^{p}(G,\R^{k})))$ with $\widehat{\nu}(\widetilde{\xi}(g))\ne 0$ for every $g\in G$ we have}
\[\sum_{g\in G}|1-\widehat{\nu}(\widetilde{\xi}(g))|<\infty.\]
Fix an enumeration $(g_{n})_{n=1}^{\infty}$ of $G$ and a permutation $\sigma\colon \N\to\N.$ For $n\in \N,$ let $E_{n}=\{g_{j}:1\leq j\leq n\},$ $F_{n}=\{g_{\sigma(j)}:1\leq j\leq n\},$ and define $\xi_{n},\zeta_{n}\in M_{1,k}(\ell^{p}(G,\R))$ by $(\xi_{n})_{1j}(g)=1_{E_{n}}(g)\widetilde{\xi}(g)(j),$ $(\zeta_{n})_{1j}(g)=\widetilde{\xi}(g)(j)1_{F_{n}}(g).$ Then,
\[\lim_{n\to\infty}\prod_{j=1}^{n}\widehat{\nu}(\widetilde{\xi}(g_{j}))=\lim_{n\to\infty}\widehat{\mu}_{\xi_{n}}(1)=\widehat{\mu}_{\xi}(1),\]
\[\lim_{n\to\infty}\prod_{j=1}^{n}\widehat{\nu}(\widetilde{\xi}(g_{\sigma(j)}))=\lim_{n\to\infty}\widehat{\mu}_{\zeta_{n}}(1)=\widehat{\mu}_{\xi}(1).\]
So
\[\lim_{n\to\infty}\prod_{j=1}^{n}\widehat{\nu}(\widetilde{\xi}(g_{\sigma(j)}))=\lim_{n\to\infty}\prod_{j=1}^{n}\widehat{\nu}(\widetilde{\xi}(g_{j}))).\]
Thus the infinite product $\prod_{g\in G}\widehat{\nu}(\widetilde{\xi}(g))$ does not depend upon how one enumerates $G,$ and so by a well known theorem we have that
\[\sum_{g\in G}|1-\widehat{\nu}(\widetilde{\xi}(g))|<\infty.\]

By  claim 1, we may find a $1\leq j\leq k$ and a sequence $t_{n}\in \R$ of nonzero real numbers with $|t_{n}|<2^{-n/p}$ so that $|1-\widehat{\nu}(t_{n}e_{j})|\geq 2^{n}|t_{n}|^{p}.$ We may also choose $t_{n}$ so that $|1-\widehat{\nu}(t_{n}e_{j})|< 1/2$ for all $n\in \N.$
Choose a sequence $(E_{n})_{n}$ of disjoint, nonempty, finite subsets of $G$ so that $\frac{2^{-n}}{|t_{n}|^{p}}\leq |E_{n}|<\frac{2^{-n}}{|t_{n}|^{p}}+1.$ Since $|t_{n}|<2^{-n/p}$ and $G$ is infinite, it is possible to choose such a sequence. Now define $\xi\in M_{1,k}(\R^{G})$ by $\xi_{1l}=\delta_{l=j}\sum_{n}t_{n}1_{E_{n}}.$ We then have that
\[\|\widetilde{\xi}\|_{p}^{p}=\sum_{n}|t_{n}|^{p}|E_{n}|\leq \sum_{n}2^{-n}+|t_{n}|^{p}\leq 2\sum_{n}2^{-n}<\infty.\]
So $\xi\in M_{1,k}(\ell^{p}(G,\R))$ and $\widehat{\nu}(\widetilde{\xi}(g))\ne 0$ for all $g\in G.$ So by claim 2,
\[\sum_{g\in G}|1-\widehat{\nu}(\widetilde{\xi}(g))|<\infty.\]
But
\[\sum_{g\in G}|1-\widehat{\nu}(\widetilde{\xi}(g))|=\sum_{n}|1-\widehat{\nu}(t_{n}e_{j})||E_{n}|\geq \sum_{n}2^{n}2^{-n}=\infty,\]
and this gives a contradiction.

\end{proof}

We also show that the assumption that $\nu$ has a finite second moment and has mean zero is necessary to extend $\Theta_{\xi}$ to $\xi\in M_{m,k}(\ell^{2}(G,\R)).$

\begin{prop}\label{P:cant extend seconds}
Let $G$ be a countably infinite group, and $k\in \N,m\in \N\cup\{\infty\}.$ Suppose that $\nu\in \Prob(\R^{k})$ and that $\nu$ is not the dirac mass at $0.$ Suppose that one of the following two conditions hold:
\begin{itemize}
    \item either $\nu$ does not have a finite second moment, or
    \item $\nu$ has a finite second moment but does not have mean zero.
\end{itemize}
\begin{enumerate}[(i)]
\item There is no extension of the map  $M_{m,k}(c_{c}(G,\R))\to \Meas((\R^{k})^{G},\nu^{\otimes G},(\T^{m})^{G}),$ $\xi\mapsto \Theta_{\xi}$ to a continuous map $M_{m,k}(\ell^{2}(G,\R))\to \Meas((\R^{k})^{G},\nu^{\otimes G},(\T^{m})^{G}).$ \label{I:cant extend map without seconds}
\item There is no  extension of the map  $M_{m,k}(c_{c}(G,\R))\to \Prob((\T^{m})^{G}),$ $\xi\mapsto (\Theta_{\xi})_{*}(\nu^{\otimes G})$ to a continuous map $M_{m,k}(\ell^{2}(G,\R))\to \Prob((\T^{m})^{G}).$ \label{I:cant probability extend with seconds}
\end{enumerate}

\end{prop}

\begin{proof}
As in the proof of Proposition \ref{P:cant extend}, we may, and will, assume that $m=1.$ As in the proof of Proposition \ref{P:cant extend}, it suffices to show (\ref{I:cant probability extend with seconds}).

We first prove the following claim.
\emph{Claim. There is a $1\leq j\leq k$ so that}
\[\lim_{t\to 0}\frac{|1-\widehat{\nu}(te_{j})|}{|t|^{2}}=\infty.\]
Suppose that the claim is false. We first show that $\nu$ has a finite second moment. We may apply Fatou's Lemma as in the proof of Proposition \ref{P:cant extend} to see that for every $1\leq j\leq k:$
\[\int |x_{j}|_{2}^{2}\,d\nu(x)=\frac{1}{\pi^{2}}\int\lim_{t\to 0}\left(\frac{\sin(\pi t x_{j})}{t}\right)^{2}\,d\nu(x)\leq \liminf_{t\to 0}\frac{1}{2\pi^{2}}\rea{\left(\frac{1-\widehat{\nu}(t e_{j})}{t^{2}}\right)}<\infty,\]
since we are assuming that
\[\liminf_{t\to 0}\frac{|1-\widehat{\nu}(te_{j})|}{|t|^{2}}<\infty.\]
Since this is true for all $j=1,\cdots,k$ we see that $\nu$ has a finite second moment.

Thus, by hypothesis, we must have that $\nu$ does not have mean zero. So we may choose a $1\leq j\leq k$ so that $\int x_{j}\,d\nu(x)\ne 0.$ Since $\nu$ has a finite second moment, we may apply the dominated convergence theorem to see that
\[\lim_{t\to 0}\frac{\widehat{\nu}(te_{j})-1}{t}=\frac{\partial \widehat{\nu}}{\partial t_{j}}(0)=2\pi i\int x_{j}\,d\nu(x)\ne 0.\]
Thus
\[\lim_{t\to 0}\frac{|1-\widehat{\nu}(te_{j})|}{t^{2}}=\infty.\]
This gives a contradiction, so we have shown the claim.

Once we have shown the claim, the proof proceeds, mutatis mutandis, as in Proposition \ref{P:cant extend}.

\end{proof}

Propositions \ref{P:cant extend},\ref{P:cant extend seconds} show that defining $\Theta_{\xi}$ in a continuous manner when $\xi\notin M_{m,k}(\ell^{2}(G,\R))$ is not possible, and that our hypothesis that $\nu\in \Prob(\R^{k})$ has mean zero and a finite second moment is necessary. Thus Corollary \ref{C:meas extens} should be taken as the optimal context in which we can measurably  extend convolution $\xi$ beyond assuming that $\xi$ is finitely supported.

Though we will not use it much, we close this section by connecting our definition of $X^{\xi}$ to the group $X_{f}$ associated to an $f\in M_{m,k}(\Z(G)).$ The action $G\actson X_{f}$ has appeared in several previous works on algebraic actions (e.g. see \cite{BowenEntropy,BowenExamples, BowenLi, Me5, LiThom, LindSchmidt2}).

\begin{defn}
Let $m\in \N\cup\{\infty\},k\in \N,$ and  $f\in M_{k,m}(\Z(G)).$ We let $X_{f}$ be the Pontryagin dual of $\Z(G)^{\oplus m}/r(f)(\Z(G)^{\oplus k}). $ It is equipped with the natural algebraic action $G\actson X_{f}$ dual to the natural action of $G$ on $\Z(G)^{\oplus m}/r(f)(\Z(G)^{\oplus k}):$
\[(gx)(a)=x(g^{-1}a)\mbox{ for $g\in G,a\in \Z(G)^{\oplus m}/r(f)(\Z(G)^{\oplus k})$, $x\in X_{f}$}.\]

\end{defn}

By Pontryagin duality, we may naturally identify $X_{f}$ with $(r(f)(\Z(G)^{\oplus k}))^{o}$ and in this way regard it as a closed subgroup of $(\T^{m})^{G}.$ We will implicitly do so for the remainder of the paper.

We will need notation for multiplication of matrices. If $f\in M_{m,n}(\R(G)),\xi\in M_{n,k}(\ell^{2}(G,\R))$ with $n\in \N,m,k\in \N\cup\{\infty\},$ then we define $f\xi\in M_{m,k}(\ell^{2}(G,\R))$ by
\[(f\xi)_{ij}=\sum_{l=1}^{n}\lambda(f_{il})\xi_{lj} \mbox{ for $1\leq i\leq m,1\leq j\leq k$.}\]
Similarly, if $f\in M_{m,n}(\R(G)),\xi\in M_{k,m}(\ell^{2}(G,\R)),$ with $m\in \N,n,k\in \N\cup\{\infty\},$ then we define $\xi f\in M_{k,n}(\R(G))$ by
\[(\xi f)_{ij}=\sum_{l=1}^{m}\lambda(\xi_{il})f_{lj},\mbox{ for $1\leq i\leq k, 1\leq j\leq n. $}\]
Here we recall our notational conventions stated at the beginning of Section \ref{S:extend convolve}.

\begin{defn}
Given $f\in M_{n}(\R(G))$ we say that $\xi\in M_{n}(\ell^{2}(G,\R))$ is an $\ell^{2}$ formal inverse to $f$ if
\[(\xi f)_{ij}=\delta_{i=j}\delta_{1}.\]
\end{defn}

It is well known that if $\xi$ is an $\ell^{2}$ formal inverse to $f,$ then $(f\xi)_{ij}=\delta_{i=j}\delta_{1}$ (see \cite[Proposition 2.2 (iii)]{MeWE}).  We leave it as an exercise to the reader to check that for $m,n\in \N,k\in \N\cup\{\infty\},$ and $f\in M_{m,n}(\R(G)),$ $\xi\in M_{k,m}(\ell^{2}(G,\R)),$ $\zeta\in M_{n,k}(\ell^{2}(G,\R)),$ $\alpha\in \R(G)^{\oplus k},\beta\in \R(G)^{\oplus m}, $
\[r(f)r(\xi)\alpha=r(\xi f)\alpha,\]
\[r(\zeta)r(f)\beta=r(f\zeta)\beta.\]

Similarly, for $m,n\in \N, k\in \N\cup\{\infty\},$ $f\in M_{m,n}(\R(G)),$ $\xi\in M_{n,k}(\ell^{2}(G,\R)),$ $\zeta\in M_{k,m}(\ell^{2}(G,\R)),$ $\alpha\in \R(G)^{\oplus k},$ $\beta\in \R(G)^{\oplus n},$
\[\lambda(f)\lambda(\xi)\alpha=\lambda(f\xi)\alpha,\]
\[\lambda(\zeta)\lambda(f)\beta=\lambda(\zeta f)\beta.\]

\begin{prop}
Let $n\in\N,$ and suppose that $f\in M_{n}(\Z(G))$ has an $\ell^{2}$ formal inverse $\xi\in M_{n}(\ell^{2}(G,\R)).$ Then,
\[X_{f}=X^{\xi^{*}}.\]

\end{prop}

\begin{proof}
By Proposition \ref{P:annhilator computation}, we have that
\[(X^{\xi^{*}})^{o}=\{\alpha\in \Z(G)^{\oplus n}:r(\xi)\alpha\in \Z(G)^{\oplus n}\}.\]
For $\alpha\in \Z(G)^{\oplus n},$ we have that
$r(f)r(\xi)\alpha=r(\xi f)\alpha.$
Since $\xi$ is a left formal inverse to $f,$
\[r(f)r(\xi)\alpha=\alpha.\]
So if $\alpha\in (X^{\xi^{*}})^{o},$ then $\alpha\in r(f)(\Z(G)^{\oplus n}).$

Conversely, suppose that $\alpha\in r(f)(\Z(G)^{\oplus n}),$ and write $\alpha=r(f)\beta$ with $\beta\in \Z(G)^{\oplus n}.$ Then
\[r(\xi)\alpha=r(\xi)r(f)\beta=r(f\xi)\beta=\beta\in \Z(G)^{\oplus n}.\]
So we have shown that $(X^{\xi^{*}})^{o}=r(f)(\Z(G)^{\oplus n}).$ By definition, $X_{f}^{o}=r(f)(\Z(G)^{\oplus n}).$ So $(X^{\xi^{*}})^{o}=X_{f}^{o},$ and thus  $X_{f}=X^{\xi^{*}}.$

\end{proof}

\section{Applications to strong soficity and completely positive entropy}\label{S:appplications}

Note that if $\xi,\zeta\in \ell^{2}(G),$ then $\xi*\zeta\in \ell^{\infty}(G),$ where
\[(\xi*\zeta)(g)=\sum_{h\in G}\xi(h)\zeta(h^{-1}g)\mbox{ for all $g\in G.$}\]
We let
\[L(G)=\{\xi\in\ell^{2}(G):\xi*\zeta\in\ell^{2}(G) \mbox{ for all $\zeta\in \ell^{2}(G)$}\}.\]
It follows from the closed graph theorem that the operator $\lambda(\xi)\colon \ell^{2}(G)\to \ell^{2}(G)$ given by $\lambda(\xi)\zeta=\xi*\zeta$ is bounded. We let $\|\xi\|_{L(G)}$ denote the operator norm of this operator. For $\xi,\zeta\in L(G),$ we will typically use $\xi\zeta$ instead of $\xi*\zeta.$ It is direct to check that for $\xi,\zeta\in L(G)$ we have that $\lambda(\xi)\lambda(\zeta)=\lambda(\xi\zeta)$ (e.g. both sides agree on $c_{c}(G)$ and so the result follows by continuity), and by definition we have that $\xi=\lambda(\xi)\delta_{1}.$ From this, it follows that  $(\xi \zeta)\eta=\xi(\zeta \eta)$ for $\xi,\zeta,\eta\in L(G).$ It is direct to show that for $\xi\in L(G)$ we have that $\lambda(\xi)^{*}=\lambda(\xi^{*}),$ where $\xi^{*}\in \ell^{2}(G)$ is defined by $\xi^{*}(g)=\overline{\xi(g^{-1})}.$ Thus $\xi\in L(G)$ implies that $\xi^{*}\in L(G),$ and  $\|\xi^{*}\|_{L(G)}=\|\xi\|_{L(G)}.$ Since $(\xi*\zeta)^{*}=\zeta^{*}*\xi^{*}$ for all $\xi,\zeta\in \ell^{2}(G),$ we see that if $\xi\in L(G),$ then for all $\zeta\in \ell^{2}(G)$ we have that $\zeta*\xi\in \ell^{2}(G)$ and $\|\zeta*\xi\|_{2}\leq \|\xi\|_{L(G)}\|\zeta\|.$  We let $L_{\R}(G)=L(G)\cap \ell^{2}(G,\R),$ it is easy to see that $L_{\R}(G)$ is a real subalgebra of $L(G).$

If $m,n\in \N,$ and $\xi\in M_{m,n}(L(G)),$ we can define a bounded operator $\lambda(\xi)\colon \ell^{2}(G)^{\oplus n}\to \ell^{2}(G)^{\oplus m}$ via the canonical identification
\[B(\ell^{2}(G)^{\oplus n},\ell^{2}(G)^{\oplus m})\cong M_{m,n}(B(\ell^{2}(G))).\]
We let $\|\xi\|_{M_{m,n}(L(G))}$ be the norm of this operator.

By the same methods as in \cite[Proposition 2.2. (ii)]{MeWE}, for every $m,n,k\in \N,$ every $\xi\in M_{m,n}(L(G))$ gives a bounded operator
\[\lambda(\xi):M_{n,k}(\ell^{2}(G))\to M_{m,k}(\ell^{2}(G))\]
by
\[(\lambda(\xi)\zeta)_{ij}=\sum_{l=1}^{n}\lambda(\xi_{il})\zeta_{lj}.\]
Further,
\[\|\lambda(\xi)\|_{B(M_{n,k}(\ell^{2}(G)),M_{m,k}(\ell^{2}(G)))}=\|\xi\|_{M_{m,n}(L(G))}.\]
Similarly, we define $r(\xi)\in B(\ell^{2}(G)^{\oplus m},\ell^{2}(G)^{\oplus n})$ by
\[(r(\xi)\zeta)(j)=\sum_{l=1}^{m}\zeta(l)*\xi_{lj}\mbox{ for all $1\leq j\leq n$}.\]
Here we recall that we showed earlier in this section that $\zeta*\eta\in \ell^{2}(G)$ for all $\zeta\in \ell^{2}(G),\eta\in L(G).$
As above, one can show that 
\[\|r(\xi)\|_{B(\ell^{2}(G)^{\oplus m},\ell^{2}(G)^{\oplus n})}=\|\xi\|_{M_{m,n}(L(G))}.\]
We will also use $\delta_{1}\otimes \id\in M_{n}(\ell^{2}(G))$ for the matrix $(\delta_{1}\otimes \id)_{ij}=\delta_{i=j}\delta_{1}.$ With this notation, we may recover $\xi$ from $\lambda(\xi)$ by
\[\xi=\lambda(\xi)(\delta_{1}\otimes \id).\]

Note that, by definition, $L(G)$ is a subset of $\ell^{2}(G).$ We remind the reader of the notation $\|\xi\|_{2}$ for $\xi\in M_{m,n}(\ell^{2}(G))$ stated in the beginning of Section \ref{S:extend convolve}. With this in mind, we can define an appropriate notion of approximate inverses.

\begin{defn}
Let $G$ be a countable discrete group, and let $f\in M_{m,n}(\R(G)).$ Given $(\xi_{k})\in M_{n,m}(L_{\R}(G))$ we say that $(\xi_{k})_{k}$ is:
\begin{itemize}
    \item an \emph{approximate left inverse} to $f$ if
    \[\lim_{k\to\infty}\|\xi_{k}f-\delta_{1}\otimes \id\|_{2}=0,\mbox{ and } \sup_{k}\|\xi_{k}f\|_{M_{n}(L(G))}<\infty,\]
    \item  an \emph{approximate right inverse} to $f$ if
    \[\lim_{k\to\infty}\|f\xi_{k}-\delta_{1}\otimes \id\|_{2}=0,\mbox{ and } \sup_{k}\|f\xi_{k}\|_{M_{m}(L(G))}<\infty,\]
    \item an \emph{approximate inverse} if and only if it is both an approximate left inverse and an approximate right inverse.
\end{itemize}

\end{defn}

We recall the polar decomposition of a bounded operator $T\colon \mathcal{H}\to \mathcal{K}$ between Hilbert spaces $\mathcal{H},\mathcal{K}.$ We may write $T=U|T|$ where $|T|=(T^{*}T)^{1/2},$ and $U^{*}U=\Proj_{\ker(T)^{\perp}},UU^{*}=\Proj_{\overline{\Im{T}}}.$

By \cite[Proposition 13.3 (d)]{ConwayOT}, if $\xi\in M_{m,n}(L(G)),$ and $\lambda(\xi)=\overline{u}|\lambda(\xi)|$ is the polar decomposition, then $\overline{u}=\lambda(u)$ for some $u\in M_{m,n}(L(G)),$ and $|\lambda(\xi)|=\lambda(\zeta)$ for some $\zeta\in M_{n}(L(G)).$ In fact,
\[u=\overline{u}\cdot (\delta_{1}\otimes id),\]
\[\zeta=|\lambda(\xi)|\cdot (\delta_{1}\otimes \id).\]
Additionally, by \cite[Proposition 13.3 (a)]{ConwayOT} if $\phi\colon [0,\infty)\to \C$ is any bounded, Borel function then $\phi(|\lambda(\xi)|)=\lambda(\zeta_{\phi})$ for some $\zeta_{\phi}\in M_{n}(L(G)).$ Finally, we have that $u\in M_{m,n}(L_{\R}(G)),\zeta,\zeta_{\phi}\in M_{n}(L_{\R}(G))$ if $\xi\in M_{m,n}(L_{\R}(G)),$ and if $\phi$ is real valued. As above, we have that
\[\zeta_{\phi}=\phi(|\lambda(\xi)|)\cdot (\delta_{1}\otimes \id).\]

\begin{lem}\label{L:ideal test}
Let $G$ be a countable, discrete, group and $m,n\in \N.$ Fix an $f\in M_{m,n}(\R(G)).$
\begin{enumerate}[(a)]
\item If $\xi_{k}\in M_{n,m}(L_{\R}(G))$ is an approximate left inverse to $f,$ then for all $\zeta\in \ell^{2}(G)^{\oplus n}$ we have \label{I:SOT convergence 1}
\[\|r(f)r(\xi_{k})\zeta-\zeta\|_{2}\to 0.\]
\item If $\xi_{k}\in M_{n,m}(L_{\R}(G))$ is an approximate right inverse to $f,$ then for every $\zeta\in \ell^{2}(G)^{\oplus m}$ we have \label{I:SOT convergence 2}
\[\|r(\xi_{k})r(f)\zeta-\zeta\|_{2}\to 0.\]
\item Suppose $(\xi_{k})_{k}$ is an approximate left inverse to $f$. If $\alpha\in \R(G)^{\oplus n}$ and $\alpha\notin r(f)(\ell^{2}(G,\R)^{\oplus m}),$ then
\[\lim_{k\to\infty}\|r(\xi_{k})\alpha\|_{2}=\infty.\] \label{I:ideal test 1}

\end{enumerate}
\end{lem}

\begin{proof}

(\ref{I:SOT convergence 1}): If $\zeta\in \C(G)^{\oplus n},$ then simple estimates show that
\[\|r(f)r(\xi_{k})\zeta-\zeta\|_{2}\leq \sqrt{n} \left(\max_{1\leq j\leq n}\|\zeta_{j}\|_{1}\right)\|\xi_{k}f-\delta_{1}\otimes \id\|_{2}\to_{k\to\infty}0.\]
Since $\|r(f)r(\xi_{k})\|=\|\xi_{k}f\|_{M_{n}(L(G))},$ we have that $\sup_{k}\|r(f)r(\xi_{k})\|<\infty.$  Hence the case of $\zeta\in \ell^{2}(G)^{\oplus m}$ follows from the case of $\zeta\in \C(G)^{\oplus m}$ by the density of $\C(G)^{\oplus n}$ in $\ell^{2}(G)^{\oplus n}.$

(\ref{I:SOT convergence 2}): This proved in the exact same manner as (\ref{I:SOT convergence 1}).

(\ref{I:ideal test 1}):
We prove the contrapositive. So assume that $\|r(\xi_{k})\alpha\|_{2}$ does not converge to $\infty.$ Then by passing to a subsequence we may, and will, assume that there is a constant $C>0$ so that $\|r(\xi_{k})\alpha\|_{2}\leq C.$ By further passing to a subsequence we may, and will, assume that there is a $\zeta\colon G\to \R^{m}$ so that $r(\xi_{k})\alpha\to_{k\to\infty} \zeta$ pointwise. By Fatou's lemma,
\[\|\zeta\|_{2}\leq \liminf_{k\to\infty}\|r(\xi_{k})\alpha\|_{2}\leq C,\]
so $\zeta\in \ell^{2}(G,\R)^{\oplus m}.$ Moreover, since $r(\xi_{k})\alpha\to \zeta$ pointwise, we have that $r(f)r(\xi_{k})\alpha\to r(f)\zeta$ pointwise. But $r(f)r(\xi_{k})\alpha=r(\xi_{k}f)\alpha,$ and so by (\ref{I:SOT convergence 1}) we have that $\|\alpha-r(f)r(\xi_{k})\alpha\|_{2}\to 0.$ Hence $r(f)r(\xi_{k})\alpha\to \alpha$ pointwise, and so $\alpha=r(f)\zeta.$

\end{proof}

Because of the preceding proposition, it is worthwhile to address when $f$ has approximate left (or right) inverses.
\begin{prop}\label{P:approximation nonsense}
Let $G$ be a countable, discrete, group and $m,n\in \N.$ Fix an $f\in M_{m,n}(\R(G)).$ Then:
\begin{enumerate}[(a)]
\item $f$ has an approximate left inverse if and only if $\lambda(f)$ is injective, \label{I:injective approximation}
\item $f$ has an approximate right inverse if and only if $\lambda(f)$ has dense image, \label{I:dense image approximation}
\item $f$ has an approximate inverse if and only if $m=n$ and $\lambda(f)$ is injective. \label{I:inverse approximation}
\end{enumerate}

\end{prop}

\begin{proof}

(\ref{I:injective approximation}): First suppose that $f$ has an approximate left inverse $(\xi_{k})_{k}$ to $f.$ Let $\zeta\in \ker(\lambda(f)).$ By the same arguments as in Lemma \ref{L:ideal test} (\ref{I:SOT convergence 1}), we have that
\[\|\zeta\|=\|\lambda(\xi_{k})\lambda(f)\zeta-\zeta\|_{2}\to_{k\to\infty}0,\]
and thus $\zeta=0.$

Conversely, suppose that $\lambda(f)$ is injective. Let $\lambda(f)=v|\lambda(f)|$ be the polar decomposition of $\lambda(f).$ Let $\phi_{k}\colon [0,\infty)\to [0,\infty)$ be given by $\phi_{k}(t)=1_{[1/k,\infty)}(t)t^{-1},$ and set $\xi_{k}=\phi_{k}(|\lambda(f)|)v^{*}(\delta_{1}\otimes \id).$ Since $v^{*}v=\Proj_{\ker(\lambda(f))^{\perp}}=1,$ we have that $\xi_{k}f=1_{[1/k,\infty)}(|\lambda(f)|)(\delta_{1}\otimes \id).$ Hence
\[\|\xi_{k}f\|_{M_{n}(L(G))}=\|1_{[1/k,\infty)}(|\lambda(f)|)\|\leq 1.\]
Since $\lambda(f)$ is injective, we have that $1_{\{0\}}(|\lambda(f)|)=\Proj_{\ker(\lambda(f))}=0.$ Thus:
\[\|f\xi_{k}-\delta_{1}\otimes \id\|_{2}=\|1_{[0,1/k)}(|\lambda(f)|)\delta_{1}\otimes \id\|_{2}\to\|1_{\{0\}}(|\lambda(f)|)(\delta_{1}\otimes \id)\|_{2}=0,\]
by the spectral theorem.

(\ref{I:dense image approximation}):
Since $\lambda(f^{*})=\lambda(f)^{*},$ we have that $\lambda(f)$ has dense image if and only if $\lambda(f^{*})$ is injective. It is easy to see that $f$ has an approximate right inverse if and only if $f^{*}$ has an approximate right inverse, so this follows from (\ref{I:injective approximation}).

(\ref{I:inverse approximation}): If $f$ has an approximate inverse, then by (\ref{I:dense image approximation}), (\ref{I:injective approximation}), and \cite[Lemma 1.13]{Luck}, we have that $m=n$ and $\lambda(f)$ is injective.

Conversely, suppose that $m=n$ and $\lambda(f)$ is injective. Then by \cite[Lemma 1.13]{Luck}, we have that $\lambda(f)$ has dense image. Let $\lambda(f)=u|\lambda(f)|$ be the polar decomposition. Then
\[u^{*}u=\Proj_{\ker(\lambda(f))^{\perp}}=1,\,\,\, uu^{*}=\Proj_{\overline{\Im(\lambda(f))}}=1,\]
so $u$ is a unitary. Define $\phi_{k}$ as in (\ref{I:injective approximation}) and set $\xi_{k}=\phi_{k}(|\lambda(f)|)u^{*}(\delta_{1}\otimes \id).$  As in (\ref{I:injective approximation}),
\[\|\xi_{k}f-\delta_{1}\otimes \id\|_{2}\to 0.\]
Additionally, $f\xi_{k}=u1_{[1/k,\infty)}(|\lambda(f)|)u^{*}(\delta_{1}\otimes \id).$ As in (\ref{I:injective approximation}),
\[\|f\xi_{k}-u1_{(0,\infty)}(|\lambda(f)|)u^{*}\delta_{1}\otimes \id\|_{2}\to 0.\]
As in (\ref{I:injective approximation}), we know that $1_{(0,\infty)}(|\lambda(f)|)=1$ and since $u$ is a unitary this implies
\[\|f\xi_{k}-\delta_{1}\otimes \id\|_{2}\to 0.\]
\end{proof}

\begin{lem}\label{L:weak weak equiv}
Let $G$ be a countable, discrete, group, $n\in \N,$ and $f\in M_{n}(\Z(G)).$ Suppose that $\lambda(f)$ is injective.  Then there is a $\mu\in\Prob_{G}(X_{f})$ so that $G\actson (X_{f},\mu)$ is weakly contained in a Bernoulli shift and so that $|\widehat{\mu}(\alpha)|<1$ for all $\alpha\in \Z(G)^{\oplus n}\setminus r(f)(\Z(G)^{\oplus n}).$
\end{lem}

\begin{proof}
By Proposition \ref{P:approximation nonsense} (\ref{I:inverse approximation}), we may find an approximate inverse $(\xi_{k})_{k}$ to $f.$
Choose a probability measure $\eta\in \Prob(\Z^{n})$ with mean zero, finite second moment and so that $|\widehat{\eta}(x)|<1$ for every $x\in \R^{n}\setminus \Z^{n}.$ E.g., take
\[\eta=\left(\frac{1}{3}\right)^{n}\left(\sum_{l\in \Z}2^{-|l|}\delta_{l}\right)^{\otimes n}.\]
For $\delta>0,$ let $\gamma_{\delta}$ be the Gaussian measure on $\R^{n}$ uniquely characterized by
\[\widehat{\gamma}_{\delta}(t)=\exp(-\delta\|t\|_{2}^{2}).\]
Now set $\nu_{\delta}=\eta*\gamma_{\delta},$ and let $\Theta_{\xi_{k},\delta}$ be the map constructed in Corollary \ref{C:meas extens} corresponding to $\nu=\nu_{\delta}.$ Let
$\mu_{k,\delta}=(\Theta_{\xi_{k},\delta})_{*}(\nu_{\delta}^{\otimes G}).$
By Theorem \ref{T:pushforward props} (\ref{I:Fourier transform formula equiv}), for all $\alpha\in \Z(G)^{\oplus n}:$
\[\widehat{\mu}_{k,\delta}(\alpha)=\prod_{g\in G}\widehat{\eta}((r(\xi_{k})\alpha)(g))\prod_{g\in  G}\exp(-\delta\|(r(\xi_{k})\alpha)(g)\|_{2}^{2})=\exp(-\delta\|r(\xi_{k})\alpha\|_{2}^{2})\prod_{g\in G}\widehat{\eta}((r(\xi_{k})\alpha)(g)).\]

We start with two claims.

\emph{Claim 1. For every $\sigma\in r(f)(\Z(G)^{\oplus n})$}
\[\lim_{\delta\to 0}\lim_{k\to \infty}\widehat{\mu}_{k,\delta}(\sigma)=1.\]

\emph{Claim 2. For every $\alpha\in \Z(G)^{\oplus n}\setminus r(f)(\Z(G)^{\oplus n})$}
\[\limsup_{\delta\to 0}\limsup_{k\to\infty}|\widehat{\mu}_{k,\delta}(\alpha)|<1.\]

To prove claim 1, let $\beta\in \Z(G)^{\oplus n}$ be such that $\sigma=r(f)\beta.$ Then by Lemma \ref{L:ideal test},
\[\|\beta-r(\xi_{k})\sigma\|_{2}\to 0.\]
Hence, by Lemma \ref{L:L2 product computation} we have that
\[\lim_{k\to\infty}\widehat{\mu}_{k,\delta}(\sigma)=\exp(-\delta\|\beta\|_{2}^{2})\prod_{g\in G}\widehat{\eta}(\beta(g))=\exp(-\delta\|\beta\|_{2}^{2}),\]
the last line following as $\eta\in \Prob(\Z^{k})$ and $\beta\in \Z(G)^{\oplus k}.$ Letting $\delta\to 0$ proves claim 1.

To prove claim 2, there are two cases. In the first case, suppose that $\alpha\in r(f)(\ell^{2}(G,\R)^{\oplus n}),$ and choose $\zeta\in \ell^{2}(G,\R)^{\oplus n}$ so that $\alpha=r(f)\zeta.$
By Lemma \ref{L:ideal test} (\ref{I:SOT convergence 2})  we have that $\|r(\xi_{k})\alpha-\zeta\|_{2}\to 0,$ and so Lemma \ref{L:L2 product computation} implies that
\[\lim_{k\to\infty}\widehat{\mu}_{k,\delta}(\alpha)=\exp(-\delta\|\zeta\|_{2}^{2})\prod_{g\in G}\widehat{\eta}(\zeta(g)).\]
Since $\alpha\notin r(f)(\Z(G)^{\oplus n}),$ we may find a $g_{0}\in G$ so that $\zeta(g_{0})\notin \Z^{n}.$ So
\[\lim_{k\to\infty}|\widehat{\mu}_{k,\delta}(\alpha)|\leq \exp(-\delta\|\zeta\|_{2}^{2})|\widehat{\eta}(\zeta(g_{0}))|.\]
Thus
\[\limsup_{\delta\to 0}\lim_{k\to\infty}|\widehat{\mu}_{k,\delta}(\alpha)|\leq |\widehat{\eta}(\zeta(g_{0}))|<1.\]

For the second case, suppose that $\alpha\notin r(f)(\ell^{2}(G,\R)^{\oplus n}).$ Then by Lemma \ref{L:ideal test} (\ref{I:ideal test 1}),
$\|r(\xi_{k})\alpha\|_{2}\to \infty.$
As
\[|\widehat{\mu}_{k,\delta}(\alpha)|\leq \exp(-\delta\|r(\xi_{k})\alpha\|_{2}^{2}),\]
we have
$\lim_{k\to\infty}| \widehat{\mu}_{k,\delta}(\alpha)|=0.$
Thus,
\[\limsup_{\delta\to 0}\lim_{k\to\infty}| \widehat{\mu}_{k,\delta}(\alpha)|=0\]
and we have shown claim 2 in this case as well.

Let $\delta_{k}$ be any decreasing sequence of positive real numbers tending to zero. By claims 1 and 2, and a diagonal argument we may choose a strictly increasing sequence of integers $l_{k}$ so that
\begin{itemize}
    \item $\lim_{k\to\infty}\widehat{\mu}_{l_{k},\delta_{k}}(\sigma)=1$ for every $\sigma\in r(f)(\Z(G)^{\oplus n}),$
    \item $\limsup_{k\to\infty}|\widehat{\mu}_{l_{k},\delta_{k}}(\alpha)|<1$ for every $\alpha\in \Z(G)^{\oplus n}\setminus r(f)(\Z(G)^{\oplus n}).$
\end{itemize}

Set $m_{k}=\mu_{k,\delta_{k}}.$ We may choose a subsequence $(m_{k_{l}})_{l}$ and a $\mu\in \Prob_{G}((\T^{n})^{G})$ so that $m_{k_{l}}\to_{l\to\infty} \mu$ in the weak$^{*}$ topology. By construction, $\mu$ is weak$^{*}$ limit of factors of Bernoulli measures, and so $G\actson ((\T^{n})^{G},\mu)$ is weakly contained in a Bernoulli shift. By the first item above, we have that $\widehat{\mu}(\sigma)=1$ for every $\sigma\in r(f)(\Z(G)^{\oplus n}).$ Thus $\mu$ is supported on $X_{f}.$ By the second item above, we also have that $|\widehat{\mu}(\alpha)|<1$ for all $\alpha\in \Z(G)^{\oplus n}\setminus r(f)(\Z(G)^{\oplus n}).$

\end{proof}

\begin{thm}\label{T:Weak equivalence restated}
Let $G$ be a countable, discrete, group, $n\in \N,$ and $f\in M_{n}(\Z(G)).$ Suppose that $\lambda(f)$ is injective. Then $G\actson (X_{f},m_{X_{f}})$ is weakly contained in a Bernoulli shift.
\end{thm}

\begin{proof}
By Corollary \ref{C:structure corollary} we may choose  a largest closed, $G$-invariant subgroup $Y$ of $X_{f}$ with the property that $G\actson (Y,m_{Y})$ is weakly contained in a Bernoulli shift.  We prove that $Y=X_{f},$ to do this it is enough to show that $Y^{o}=r(f)(\Z(G)^{\oplus n}).$ Let $\mu$ be as in Lemma \ref{L:weak weak equiv}, and let $\alpha\in Y^{o}.$ By Corollary \ref{C:structure corollary} (\ref{I:minimal property}), we may choose an $x\in X_{f}$ so that $\supp(\mu)\subseteq x+Y.$ As $\alpha\in Y^{o},$ we have that $\ip{y,\alpha}_{\T}=\ip{x,\alpha}_{\T}$ for all $y\in x+Y,$ so $|\widehat{\mu}(\alpha)|=|\exp(2\pi i\ip{x,\alpha})|=1.$ Thus by Lemma \ref{L:weak weak equiv} we know $\alpha\in r(f)(\Z(G)^{\oplus n}).$ So  $Y^{o}=r(f)(\Z(G)^{\oplus n}).$
\end{proof}

Having proved that some balanced algebraic actions are weakly contained in Bernoulli shifts, we turn to showing that algebraic actions with dense square summable homoclinic groups are weakly equivalent to Bernoulli shifts and have completely positive entropy in the presence.

The first result we will need to show that algebraic actions with dense square summable homoclinic points have completely positive entropy in the presence is the following lemma, which may be regarded as an analogue of Lemma \ref{L:weak weak equiv}.

\begin{lem}\label{L:strong weak equiv}
Let $G$ be countable, discrete group, and $k\in \N,m\in \N\cup\{\infty\}.$ Fix $\xi_{1},\cdots,\xi_{k}\in \ell^{2}(G,\R)^{m}.$ Define $\xi\in M_{k,m}(\ell^{2}(G,\R))$ by $\xi_{ij}(g)=\xi_{i}(g)(j)$ for $1\leq i\leq k,1\leq j\leq m,g\in G.$
Then there is a sequence $(\nu_{n})_{n}$ in $\Prob(\Z^{k})$ with the following property. If we let $\Theta_{\xi^{*}}^{\nu_{n}}$ denote the map constructed in Corollary \ref{C:meas extens} for $\nu=\nu_{n},$ then
\[m_{X^{\xi}}=\textnormal{weak}^{*}-\lim_{n\to\infty}(\Theta_{\xi^{*}}^{\nu_{n}})_{*}(\nu_{n}^{\otimes G}).\]
\end{lem}

\begin{proof}
Let $u_{\{-n,\cdots,n\}}$ be the uniform measure on $\{-n,\cdots,n\}$ and set $\nu_{n}=u_{\{-n,\cdots,n\}}^{\otimes k}.$ Then $\widehat{\nu}_{n}(x)=1$ for all $x\in \Z^{k}$ and all $n\in \N,$ and $|\widehat{\nu}_{n}(x)|\to_{n\to\infty}0$ for all $x\in \R^{k}\setminus \Z^{k}.$ Set $\mu_{n,\xi}=(\Theta_{\xi^{*}}^{\nu_{n}})_{*}(\nu_{n}^{\otimes G}).$

Let $\alpha\in (X^{\xi})^{o}.$ Then by Proposition \ref{P:annhilator computation}, we have that $r(\xi^{*})\alpha\in c_{c}(G,\Z^{k}).$ Hence
\[\widehat{\mu}_{n,\xi}(\alpha)=\prod_{g\in G}\widehat{\nu}_{n}((r(\xi^{*})\alpha)(g))=1,\]
for every $n\in \N.$ Now suppose that $\alpha\in \Z(G)^{\oplus m}$ and that $\alpha\notin (X^{\xi})^{0}.$ Then by Proposition \ref{P:annhilator computation} we may find a $g_{0}\in G$ so that $(r(\xi^{*})\alpha)(g_{0})\notin \Z^{k}.$ Thus,
\[|\widehat{\mu}_{n,\xi}(\alpha)|=\prod_{g\in G}|\widehat{\nu}_{n}((r(\xi^{*})\alpha)(g))|\leq |\widehat{\nu}_{n}((r(\xi^{*})\alpha)(g_{0}))|\to_{n\to\infty}0.\]
Thus
\[\lim_{n\to\infty}\widehat{\mu}_{n,\xi}=1_{(X^{\xi})^{0}}=\widehat{m}_{X^{\xi}}.\]
Since the Fourier transform is a homeomorphism onto its image, this means that
\[\lim_{n\to\infty}\mu_{n,\xi}=m_{X^{\xi}}.\]

\end{proof}

We remark that Lemma \ref{L:strong weak equiv} ends up being a rather strong analogue of Lemma \ref{L:weak weak equiv}. Let $\xi$ be as in Lemma \ref{L:strong weak equiv} and $f$ as in Lemma \ref{L:weak weak equiv}. Observe that $X_{f},X^{\xi}$ are closed, $G$-invariant subgroups of $(\T^{n})^{G},(\T^{m})^{G}.$ The proof of Theorem \ref{T:Weak equivalence restated} shows that $m_{X_{f}}$ is a limit of measures \emph{supported on $(\T^{n})^{G}$} which are factors of Bernoulli measures. This is equivalent to being weakly contained in a Bernoulli shift. However, Lemma \ref{L:strong weak equiv} shows that $m_{X^{\xi}}$ is a limit of measures \emph{supported on $X^{\xi}$} (and not just $(\T^{m})^{G}$) which are factors of Bernoulli measures. This seems at first like a mild strengthening. The remarkable fact is that is a rather strong difference, as Lemma \ref{L:strong weak equiv} in fact implies that $G\actson (X^{\xi},m_{X^{\xi}})$ has completely positive entropy in the presence in addition to being weakly contained in a Bernoulli shift.
\begin{thm}\label{T:square summable cpe weak contain}
Let $G$ be countably infinite, discrete group, and let $X$ be a compact, abelian group. Suppose that $\Delta^{(2)}(G\actson X)$ is dense. Then:
\begin{enumerate}[(i)]
\item the kernel $N$ of $G\actson X$ is finite, and the induced action $G/N\actson (X,m_{X})$ is a free action, \label{I:kernel square homoclinic}
\item  the induced action $G/N\actson (X,m_{X})$ is weakly equivalent to a Bernoulli shift, \label{I:WCL2Bern}
\item if $G$ is sofic, then $G\actson (X,m_{X})$ has completely positive entropy in the presence with respect to any sofic approximation of $G.$ \label{I:CPEL2Bern}
\end{enumerate}

\end{thm}

\begin{proof}

We may always embed $X$ as a $G$-invariant, closed subgroup of $(\T^{\N})^{G}.$ We fix such an embedding for the remainder of the proof. Since $\Delta^{(2)}(G\actson X)$ is dense and $\Delta^{(2)}(G\acston (\T^{N})^{G})=q(\ell^{2}(G,\R)^{\infty}),$ we may find a sequence $(\xi_{j})_{j=1}^{\infty}\in \ell^{2}(G,\R)^{\infty}$ so that $X=\bigvee_{j}X^{\xi_{j}}.$

(\ref{I:kernel square homoclinic}): 
By \cite[Proposition 4.6]{BowenLi} the fact that $\Delta^{(2)}(G\actson X)$ is dense implies that $G\actson (X,m_{X})$ is mixing. Hence by \cite{RobinAF} we know that there is a finite $N\triangleleft G$ so that $\Stab(x)=N$ for almost every $x\in X.$ So the action $G\actson (X,m_{X})$ induces a free action $G/N\actson (X,m_{X}).$

(\ref{I:WCL2Bern}):
By (\ref{I:kernel square homoclinic}) we know that  $G/N\actson (X,m_{X})$ is free.
Hence, by \cite{AbertWeiss} we know that $G/N\actson (X,m_{X})$ weakly contains a Bernoulli shift. So it simply suffices to show that $G/N\actson (X,m_{X})$ is weakly contained in a Bernoulli shift. 

We claim that in order to show that  $G/N\actson (X,m_{X})$ is weakly contained in a Bernoulli shift, it suffices to show that $G\actson (X,m_{X})$ is weakly contained in a Bernoulli shift. Suppose that $G\actson (X,m_{X})$ is weakly contained in $G\actson (A,\alpha)^{G}$ for some probability space $(A,\alpha).$ We may assume that $A$ is a compact, metrizable space and that $\alpha$ is a completed Borel probability measure on $A.$ Let $B$ be $A^{N}$ modulo the equivalence relation  $a_{1}\thicksim a_{2}$ if there is a $x\in N$ with $a_{1}=xa_{2}.$ Since $N$ is finite, it is easy to see that $B$ is a compact, metrizable space in the quotient topology. Let $\nu$ be completion of the Borel measure which is the pushforward of $\mu^{\otimes N}$ under the natural quotient map $A^{N}\to B.$
Observe that the action $G\actson (B,\nu)^{G/N}$ is precisely the factor of the action $G\actson (A,\alpha)^{G}$ corresponding to the $G$-invariant, complete, $\sigma$-algebra of sets which are $N$-invariant modulo null sets. Since $N$ is finite, and  $N\actson (X,m_{X})$ trivially, it is easy to see that the assumption that $G\actson (X,m_{X})$ is weakly contained in $G\actson (A,\alpha)^{G}$ implies that $G/N\actson (X,m_{X})$ is weakly contained in $G/N\actson (B,\nu)^{G/N}.$ So it suffices to show that $G\actson (X,m_{X})$ is weakly contained in a Bernoulli shift.

By Corollary \ref{C:structure corollary}, it suffices to show that each $G\actson (X^{\xi_{j}},m_{X^{\xi_{j}}})$ is weakly contained in a Bernoulli shift. But this is obvious from Lemma \ref{L:strong weak equiv}.

(\ref{I:CPEL2Bern}): Now assume that $G$ is sofic and fix a sofic approximation $\sigma_{k}\colon G\to S_{d_{k}}.$ By part (\ref{I:WCL2Bern}), and \cite[Corollary 3.6]{Me13}, we know that $G\actson X$ is strongly sofic with respect to $(\sigma_{k})_{k}.$ By Corollary \ref{C:generating  by cpe subgroups}, it suffices to show that $G\actson (X^{\xi_{j}},m_{X^{\xi_{j}}})$ has completely positive entropy in the presence for each $j\in \N.$ So it suffices to show that $G\actson (X^{\xi},m_{X^{\xi}})$ has completely positive entropy in the presence for each $\xi\in \ell^{2}(G,\R)^{\infty}.$

 By (\ref{I:WCL2Bern}), we know that $G\actson (X^{\xi},m_{X^{\xi}})$ is strongly sofic with respect to $(\sigma_{k})_{k},$ so its outer Pinsker factor is given by
$G\actson (X^{\xi}/Y,m_{X^{\xi}/Y})$
for some $G$-invariant, closed, normal subgroup $Y,$ and the factor map is simply  the quotient map $q_{Y}.$
Let $\nu_{n}$ be as in Lemma \ref{L:strong weak equiv}. Then (with limits taken in the weak$^{*}$ topology):
\[m_{X^{\xi}}=\lim_{n\to\infty}(\Theta_{\xi^{*}}^{\nu_{n}})_{*}(\nu_{n}^{\otimes G}).\]
For each $n\in \N,$ let
$\zeta_{n}=(\Theta_{\xi^{*}}^{\nu_{n}})_{*}(\nu_{n}^{\otimes G}).$
Since $\Theta_{\xi^{*}}^{\nu_{n}}$ is almost surely valued in $X^{\xi}$ by Theorem \ref{T:pushforward props} (\ref{I:controlling the image of conv}), we have that $\zeta_{n}$ is a probability measure on $X^{\xi}.$ So we can  define a probability measure $\eta_{n}$ on $X^{\xi}/Y$ by
$\eta_{n}=(q_{Y})_{*}(\zeta_{n}).$
Observe that
\begin{equation}\label{E:limits of points are points}
m_{X^{\xi}/Y}=\lim_{n\to\infty}\eta_{n}.
\end{equation}

Fix an $n\in \N.$ Then,
\begin{align*}
   h_{(\sigma_{k})_{k}}(G\actson (X^{\xi}/Y,\eta_{n}):(X^{\xi},\zeta_{n}))&\leq h_{(\sigma_{k})_{k},\topo}(G\actson X^{\xi}/Y:X^{\xi})\\
&=h_{(\sigma_{k})_{k}}(G\actson (X^{\xi}/Y,m_{X^{\xi}/Y}):(X^{\xi},m_{X^{\xi}}))
\end{align*}
the first inequality being obvious from the definitions, and the last equality following from strong soficity of $G\actson (X^{\xi},m_{X^{\xi}})$ and  \cite[Theorem 1.1]{Me12}. By definition of the outer Pinsker factor we see that
\[h_{(\sigma_{k})_{k}}(G\actson (X^{\xi}/Y,\eta_{n}):(X^{\xi},\zeta_{n}))\leq 0.\]
But, by definition, $G\actson (X^{\xi},\zeta_{n})$ is a factor of a Bernoulli shift and thus has completely positive entropy in the presence by \cite{KerrCPE}. So we must have that $\eta_{n}$ is a  dirac mass. By equation (\ref{E:limits of points are points}), we see that $m_{X^{\xi}/Y}$ is a dirac mass, so $Y=X^{\xi}.$ The definition of the outer  Pinsker factor now shows that $G\actson (X^{\xi},m_{X^{\xi}})$ has completely positive entropy in the presence.

\end{proof}

We remark that the following follows from the proof of Theorem \ref{T:square summable cpe weak contain}.

\begin{cor}
Let $G$ be a countable, discrete, sofic group and $G\actson X$ an algebraic action. Suppose that there are Lebesgue spaces $(B_{n},\beta_{n})$ and $G$-equivariant, measurable maps $\Theta_{n}\colon B_{n}^{G}\to X$ so that
\[(\Theta_{n})_{*}(\beta_{n}^{\otimes G})\to m_{X}.\]
Then $G\actson (X,m_{X})$ is weakly contained in a Bernoulli shift, and has completely positive entropy in the presence (with respect to any sofic approximation of $G$).

\end{cor}

Another consequence of Theorem \ref{T:square summable cpe weak contain} is the following result for actions on totally disconnected groups.

\begin{cor}
Let $G$ be a countable, discrete, sofic group with sofic approximation. Let $X$ be a compact, totally disconnected, abelian group and let $G\actson X$ be an algebraic action. Suppose that the homoclinic group of $G\actson X$ is dense. Then $G\actson (X,m_{X})$ has completely positive entropy in the presence with respect to every sofic approximation of $G.$
\end{cor}

\begin{proof}
Since $X$ is totally disconnected, we know that $\widehat{X}$ is torsion, and thus the set of square-summable homoclinic points and the set of homoclinic points agree. So this follows from  Theorem \ref{T:square summable cpe weak contain}.

\end{proof}

We close by remarking that it is easy to construct many examples of algebraic actions with dense square summable homoclinic group.

\begin{defn}\label{D:homoclinc assoicated to a rep}
Let $G$ be a countable, discrete, group and $n\in \N.$ Let $\mathcal{H}\subseteq \ell^{2}(G)^{\oplus n}$ be a closed, $G$-invariant, linear subspace. We set
\[X^{\mathcal{H}}=\overline{\{q(\xi):\xi\in \mathcal{H}\}}.\]
\end{defn}

It is clear that $X^{\mathcal{H}}$ is $G$-invariant, and tautologically $\Delta^{(2)}(G\actson X^{\mathcal{H}})$ is dense in $X^{\mathcal{H}}.$ So to any closed, $G$-invariant subgroup of $\mathcal{H}$ we can naturally associate an algebra subshift of $(\T^{n})^{G}$ whose square summable homoclinic points are dense.

\appendix

\section{On invariant random subgroups of algebraic actions}\label{S:IRS}
The notion of weak containment, especially weak equivalence to Bernoulli shifts, is intimately connected with freeness of actions or more generally with invariant random subgroups (see e.g. \cite{AbertWeiss}). Let $G$ be a countable, discrete, group. We use $\Sub(G)$ for the space of subgroups of $G$ in the Chabauty topology. Suppose $G\actson (Y,\nu)$ is a probability measure-preserving action, then we have the stabilizer map $\Stab\colon Y\to \Sub(G)$ which sends each point to its stabilizer. The pushforward of $\nu$ under this map is called an \emph{invariant random subgroup of $G$}. The reason for this terminology is that $\nu$ is invariant under the conjugation action $G\actson \Sub(G).$ The term ``invariant random subgroup" was given by Ab\'{e}rt-Glasner-Virag \cite{AVY14}. Related ideas had been in the mathematical community for some time, first appearing in work of Zimmer (see \cite{StuckZimmer}). There were also similar results  before Ab\'{er}t-Glasner-Virag by Aldous-Lyons \cite{AldousLyons}, Bergeron-Gaboriau \cite{BergeronGaboriau}, and Vershik \cite{Vershik12}. The study of invariant random subgroups is a quite active area of current research, see e.g. \cite{AVY14, 7s12, StuckZimmer, BowenIRSFurst, RobinAF, BurtonKechrisMax}.

Because of our results on weak containment, it is natural to investigate invariant random subgroups coming from algebraic actions, and classify which ones can occur. Fortunately this ends up being straightforward from known results.
We say that a group is \emph{FC} if all of its conjugacy classes are finite. If $P$ is a property of groups, we say that a group $G$ is locally $P$ if all of its finitely generated subgroups satisfy $P.$

\begin{thm}\label{T:IRS}

Let $X$ be a compact group, $G$ a countable, discrete, group and $G\actson X$ a faithful algebraic action.
\begin{enumerate}[(i)]
\item If $X$ is abelian, then there is a normal, locally finite subgroup $N$ of $G$ so that for almost every $x\in X,$ we have that $\Stab(x)\leq N.$
\item In general, there is a normal, locally FC subgroup $N$ of $G$ so that for almost every $x\in X,$ we have that $\Stab(x)\leq N.$
\end{enumerate}
\end{thm}

The bulk of the proof of Theorem \ref{T:IRS} is already in \cite[Lemma 2]{FCAut}. The only really new observation is the following simple proposition.

\begin{prop}\label{P:almost everywhere containment}
Let $G$ be a countable, discrete group and $G\actson X$ an algebraic action. Let
\[N=\{g\in G:[X:\Fix_{g}(X)]<\infty\}.\]
Then $N\triangleleft G$ and for almost every $x\in X$ we have that $\Stab(x)\leq N.$
\end{prop}

\begin{proof}
It is obvious that $N$ is a normal subgroup.
Observe that for every $g\in G\setminus N,$ we have that $m_{X}(\Fix_{g}(X))=0.$
So
\[\{x\in X:\Stab(x)\nsubseteq N\}=\bigcup_{g\in G\setminus N}\Fix_{g}(X)\]
has measure zero.

\end{proof}

\begin{proof}[Proof of Theorem \ref{T:IRS}]
By Proposition \ref{P:almost everywhere containment} and the fact that $G$ acts faithfully, it suffices to show that any finitely generated subgroup of
\[\{\alpha\in \Aut(X):[X:\Fix_{\alpha}(X)]<\infty\}\]
is FC, and that if $X$ is abelian such a finitely generated subgroup is finite. This follows from \cite[Lemma 2]{FCAut}.
\end{proof}

Notice that Theorem \ref{T:IRS} is optimal. For example, if we let $S_{\infty}$ be the group of permutations of $\N$ with only finitely many non-fixed points, we then have a generalized Bernoulli action $S_{\infty}\actson (\Z/2\Z)^{\N}$ induced from the action $S_{\infty}\actson \N.$ The group $S_{\infty}$ is clearly locally finite, and the stabilizers of this action are typically infinite. 
In fact, $\{g\in S_{\infty}:[X:\Fix_{g}(X)]<\infty\}=S_{\infty}.$ So in general we cannot force $N$ in Theorem \ref{T:IRS} to be a proper subgroup of $G$ if $G$ is locally finite. The action $S_{\infty}\actson (Z/2\Z)^{\N}$ is a \emph{reduced totally non-free action} in the sense of Vershik (see \cite{Vershik12} for more details).

For certain choices of $G,$ Theorem \ref{T:IRS} implies that any ergodic, faithful algebraic action is free. To prove this we use the following simple corollary of Theorem \ref{T:IRS}.

\begin{cor}\label{C: easy IRS}
Let $G$ be a countable, discrete group, and $G\actson X$ an algebraic action such that $G\actson (X,m_{X})$ is ergodic and faithful.
\begin{enumerate}[(i)]
\item If every normal, locally FC subgroup of $G$ has only countably many subgroups, then $G\actson (X,m_{X})$ is essentially free. \label{I:FC free}
\item  If every normal, locally finite subgroup of $G$ has only countably many subgroups, and if $X$ is abelian, then $G\actson (X,m_{X})$ is essentially free.  \label{I:Locall fintie free}
\end{enumerate}

\end{cor}

\begin{proof}
We only prove (\ref{I:FC free}), the proof of (\ref{I:Locall fintie free}) is similar. Let $\nu=\Stab_{*}(m_{X}).$ By Theorem \ref{T:IRS}, we know that $\nu$ is supported on a countable set. Since $G\actson (X,m_{X})$ is ergodic, we know by \cite[Lemma 2.13]{Me12} that $G\actson (X,m_{X})$ is weakly mixing (this also essentially follows from \cite[Lemma 1.2]{Schmidt}). Thus $G\actson (\Sub(G),\nu)$ is weakly mixing, and since $\nu$ is supported on a countable set this is only possible when $\nu$ is a point mass. Hence $\Stab(x)=\{1\}$ for almost every $x\in X.$

\end{proof}

\begin{cor}
Let $G$ be a countable, discrete, group and let $G\actson X$ be a faithful algebraic action.
\begin{enumerate}[(i)]
\item If $G$ is torsion-free and $X$ is abelian, then $G\actson (X,m_{X})$ is essentially free. \label{I:torsion free freeness}
\item Suppose $G\actson (X,m_{X})$ is ergodic, and $G$ has a finite-index subgroup $H$ which is any of the following:
\begin{itemize}
    \item a finitely generated nilpotent group,
    \item a hyperbolic group,
    \item $\textnormal{Out}(\F_{n})$ for some $n\in \N,$
    \item a finitely generated right-angled Artin group, or more generally a finite graph product of free abelian groups of finite rank,
    \item a mapping class group of an orientable, connected Riemann surface with negative Euler characteristic,
    \item a group with a bound on the order of its finite subgroups and which acts properly on a $\textnormal{CAT}(0)$ cubical complex,
\end{itemize}
then $G\actson (X,m_{X})$ is essentially free.
\label{I:hyperbolic freeness}

\end{enumerate}

\end{cor}

\begin{proof}

(\ref{I:torsion free freeness}): Trivial from Theorem \ref{T:IRS}.

(\ref{I:hyperbolic freeness}):
In each case, we verify that every locally FC, normal, subgroup of $H$ has only countably many subgroups. Since $H$ is finite-index in $G,$ this implies that every locally FC, normal subgroup of $G$ has only countably many subgroups. This is trivial if $H$ is a finitely generated nilpotent group. When $H$ is a hyperbolic group, it follows from \cite[Section 8]{GromovHyper},\cite{GhyDHPHyper} and when $H$ is $\textnormal{Out}(\F_{n})$ it follows from the Tits alternative for $\textnormal{Out}(\F_{n})$  (see \cite{TitsAltBMFI, TitsAltBMFII} and \cite{TitsAltBMFIII}). For the case of mapping class groups see \cite{TitsAltMcCarthy, TitsAltIvanov, SolvableSubBLM}, see \cite{TitsAltMW} for the groups which act properly on a $\textnormal{CAT}(0)$ cubical complex, and \cite{TitsAltAntAsh} for right-angled Artin groups and graph products.

\end{proof}


%

\end{document}